%% file: mainOn.tex
\documentclass{article}%
\usepackage{amsmath}
\usepackage{amsfonts}
\usepackage[shortlabels]{enumitem}
\usepackage{amssymb}
\usepackage{amsthm}
\usepackage[toc,page]{appendix}
\usepackage{mathtools}
\usepackage{fullpage}
\usepackage{comment}
\usepackage{siunitx}
\usepackage{booktabs}
\usepackage{graphicx,color,xcolor}
\usepackage{graphicx}
\usepackage[hidelinks]{hyperref}
\usepackage[capitalize]{cleveref}
\ifpdf
  \usepackage{crossreftools}
\fi
\usepackage{tikz-cd}
\usepackage{multirow,booktabs}
\allowdisplaybreaks
\providecommand{\U}[1]{\protect\rule{.1in}{.1in}}
\crefname{equation}{}{}
\Crefname{equation}{Equation}{Equations}
\crefname{proposition}{Proposition}{Propositions}
\Crefname{proposition}{Proposition}{Propositions}
\crefname{condition}{Condition}{Conditions}
\crefname{conjecture}{Conjecture}{Conjectures}
\newtheorem{theorem}{Theorem}

\newtheorem{algorithm}[theorem]{Algorithm}

\newtheorem{assumption}[theorem]{Assumption}

\newtheorem{definition}[theorem]{Definition}
\newtheorem{example}[theorem]{Example}

\newtheorem{lemma}[theorem]{Lemma}
\newtheorem{notation}[theorem]{Notation}
\newtheorem{problem}[theorem]{Problem}
\newtheorem{proposition}[theorem]{Proposition}
\newtheorem{remark}[theorem]{Remark}

\newcommand{\RR}{\mathbb{R}}

\begin{document}

\title{Certificate for Orthogonal Equivalence of Real Polynomials\\ by Polynomial-Weighted Principal
Component Analysis\thanks{Equal contribution by all the authors. The authors are ordered alphabetically. }}

\author{Martin Helmer \\
Department of Mathematics, Swansea University  \\
Swansea, SA1 8EN Wales, UK\\
martin.helmer@swansea.ac.uk
\\[10pt]
David Hong  \\
Department of Electrical and Computer Engineering, University of Delaware \\
Newark, DE 19716, US \\
hong@udel.edu
\\[10pt]
Hoon Hong  \\
Department of Mathematics, North Carolina State University \\
Raleigh, NC 27695, US\\
hong@ncsu.edu
}
\maketitle
\begin{abstract}
    Suppose that $f(x)\in \RR[x_1,\dots, x_n]$ and $g(x)\in \RR[x_1,\dots, x_n]$ are two real polynomials of degree $d$ in~$n$ variables. If the polynomials $f$ and $g$ are the same up to orthogonal symmetry a natural question is then what element of the orthogonal group induces the orthogonal symmetry; i.e.~to find the element $R\in O(n)$ such that $f(Rx)=g(x)$. One may directly solve this problem by constructing a nonlinear system of equations induced by the relation $f(Rx)=g(x)$ along with the identities of the orthogonal group however this approach becomes quite computationally expensive for larger values of $n$ and $d$. To give an alternative and significantly more scalable solution to this problem, we introduce the concept of {\em Polynomial-Weighted Principal Component Analysis} (PW-PCA). We in particular show how PW-PCA can be effectively computed and how these techniques can be used to obtain a {\em certificate of orthogonal equivalence}, that is we find the $R\in O(n)$ such that $f(Rx)=g(x)$.
\end{abstract}

\section{Introduction}
Understanding the behavior of polynomials under group actions has been a topic of classical interest in algebra. A well developed and wide ranging framework which provides insight into this, and in particular answers if two polynomials lie in the same orbit under a group action, is provided by  Invariant Theory and is well described in a variety of books, e.g.~\cite{dolgachev2003lectures,neusel2007invariant,springer2006invariant, derksen2015computational,popov1994invariant}. Consider a group $G$ acting on the vector space~$V$ of polynomials of degree $d$ in $n$ variables with coefficients in a field $k$. From a computational perspective algorithms exist which allow us to decide if two polynomials $f,g$ are in the same $G$-orbit. Roughly speaking these algorithms compute various properties of the ring of invariants $k[V]^G$, see e.g.~\cite{derksen2015computational,sturmfels2008algorithms,kogan2023invariants}, and may be applied for many different groups $G$. While many of the algorithmic approaches focus on the case where~$G$ is a finite group, infinite groups are also considered, for example when $G$ is a linear algebraic  group, see e.g.~\cite{derksen2015computational} for a detailed discussion and \cite{InvariantRingArticle,InvariantRingSource} for a computer implementation. A complete set of invariants is one which generates the ring of  invariants $k[V]^G$. In the context of testing if two polynomials lay in the same $G$-orbit the variables in the ring $k[V]^G$ represent the coefficients $c$ of a polynomial $f$ in $V$ and a given invariant specified by $r(c)\in k[V]^G$ must have the same value when evaluated at the explicit coefficients of any two polynomials in $V$ in the same $G$-orbit.
Numerous results in Invariant Theory also exist which give closed form formulas for the invariants associated to particular groups. In the case of the orthogonal group this has been studied by several authors, see e.g.~\cite{gorlach2019rational,hubert2025algebraically} and the references therein; the current state of the art for the case of the orthogonal group $O(n)$ seems to be the recent work \cite{breloer2025rational} which gives a simple closed form expression for the rational invariants of the action of the orthogonal group on the vector space of polynomials of degree $d$ in $n$ variables for $d$ even (with the possibility of extension to the odd case). 

Let $R\bullet f$ denote the action of a group element $R$ on a polynomial $f$. If we have two polynomials~$f$ and~$g$ which are equivalent under a group action a natural follow up question, which is not directly addressed by methods of Invariant Theory, is how do we obtain a {\em certificate of equivalence}, that is how do we find an element $R\in G$ such that $R\bullet f=g$? In the present work we focus on a specialized method which provides such a certificate  for the case where $G$ is the orthogonal group $O(n)$.

Consider the problem of obtaining a certificate of equivalence for the action of the orthogonal group $G=O(n)$. Given fixed polynomials~$f$ and~$g$, both of degree $d$, in $\RR[x_1,\dots, x_n]$ one may consider a matrix $R$ in $n^2$ variables $r_{11}, \dots ,r_{nn}$ and, adding in the constraints $R^\top R=I_n$ which define the orthogonal group, may directly compute solutions (if any) to the system in variables $r_{11}, \dots ,r_{nn}$ arising by equating coefficients in the relation $R\bullet f=f(Rx)=g(x)$, see Algorithm~\ref{alg:baseline} in Section \ref{sec:performance}. When $n=1$ and for any $d$ or when $d=1$ and for any $n$ the system of equations can be solved very simply. Similarly, several other special cases for small $n$ or $d$ may be solved fairly easily. For example, when $d=2$, there exists simple known solutions in terms of the eigenvectors of the highest degree homogeneous parts (the quadratic form part) of $f$ and $g$. In general, however, solving the resulting non-linear system of equations quickly becomes challenging for large~$n$ and~$d$.

In this paper we propose a novel solution to this problem of certifying orthogonal equivalence of polynomials which combines ideas from algebra and statistics to yield an effective computational method which is far more scalable than the direct solution approach mentioned above.
Inspired by the widely used data analysis technique Principal Component Analysis (PCA), we will introduce a technique which we refer to as {\em Polynomial-Weighted Principal Component Analysis} (PW-PCA).

PW-PCA will allow us to essentially reduce the problem for arbitrary degree polynomials to that for the quadratic homogeneous polynomials.
We then employ this method to develop an algorithm to obtain a certificate of orthogonal equivalence for two polynomials.

Interestingly, the problem of obtaining a certificate of orthogonal equivalence has a natural relation to the so called {\em orthogonal Procrustes problem}, which has been heavily studied, e.g., in the context of aligning data points.
However, there are crucial differences between these problems that obstruct a straightforward extension of well-known techniques for the orthogonal Procrustes problem to our setting; see Remark \ref{remark:Procrustes} below.
Our proposed method instead exploits intrinsic properties of the geometry of the polynomials.

As mentioned above, our main focus is on obtaining a certificate of orthogonal equivalence under the assumption that $f$ and $g$ are already known to be orthogonally equivalent.
However, as a byproduct of our approach employing PW-PCA to obtain such a certificate, 
we can also easily devise an algorithm which could, in principle, be used to test if the polynomials $f$ and $g$ are orthogonally equivalent when this is not known apriori. In practice however the obvious implementation on this method requires working in algebraic number fields, hence posing problems for practical efficiency, see Remark \ref{remark:OrthEqImplment}. 
Thus we leave it for future work.

\subsubsection*{Structure of the paper}The paper is structured as follows.
In Section~\ref{sec:problem}, we precisely state the main problem of this paper.
In Section~\ref{sec:pwpca}, 
we develop a notion of polynomial-weighted form of principal component analysis (PW-PCA) 
 and provide an algorithm for computing it.
In Section~\ref{sec:coe_pwpca},  under a generic assumption, we develop the mathematical theory and an algorithm to tackle the main problem
of this paper by utilizing PW-PCA.
In Section~\ref{sec:performance}, we discuss the performance of a preliminary implementation of the algorithm.
In Appendix \ref{appendix:A}, we prove  the genericity of the assumption from Section~\ref{sec:coe_pwpca}. 

\section{Problem}
\label{sec:problem}
We now give an exact statement of our problem. We begin by recalling the definition of the orthogonal group $O(n)$ as a subgroup of the general linear group, ${\rm GL}(n,\RR)$, the multiplicative group of $n\times n$ invertible matrices with real entries. We also define the associated group action of $O(n)$ on polynomials and the notion of equivalence of polynomials under this action.

\begin{definition}
[Orthogonal equivalence]\label{def:O(n)}\ \  

\begin{enumerate}
\item (\textsf{Group}) The $n$-dimensional \emph{orthogonal group}, denoted as
$O\left(  n\right)  $, is defined by
\begin{equation}
O(n)=\left\{  R\in\mathbb{R}^{n\times n}:R^{\top}R=I_{n}\right\}  .
\end{equation}

\item (\textsf{Action}) The \emph{action} of $R\in O\left(  n\right)  $ on
$p\in$ $\mathbb{R}[x_{1},\ldots,x_{n}]$, denoted as $R\bullet p$, is defined
by%
\[
(R\bullet p)(x)=p(Rx).
\]

\item (\textsf{Equivalence}) We say that  $f,g\in\mathbb{R}[x_{1}%
,\ldots,x_{n}]$ are \emph{orthogonally equivalent, }and write \emph{ }%
$f\sim_{O(n)}g$, if
\[
\exists_{R\in O(n)}\ \ g=R\bullet f.
\]
\end{enumerate}
\end{definition}

With this definition in hand we may precisely state the problem of obtaining an element of $O(n)$ which certifies the  equivalence two polynomials.  
\begin{problem}
[Certificate for orthogonal equivalence]\label{prb:problem}\ 
A solution to the problem of finding a certificate of orthogonal equivalence  is to output an element $R\in O(n)$ given input polynomials $f$ and $g$ in $n$ variables of degree $d$ as below. \medskip

{\bf Given:} $f,g\in\mathbb{R}[x_{1},\ldots,x_{n}]$ of degree $d$ such that
$n,d\geq2$ and $f\sim_{O(n)}g$.\medskip

{\bf Compute:} $R\in O(n)$ such that $g=R\bullet f$.
\end{problem}

\medskip

We now illustrate the problem with a simple example that we will use as a running example throughout the paper.

\begin{example}
[Simple running example]\label{ex:running}\ 
Consider the two polynomials $f$ and $g$ of degree $d=3$ in $\RR[x_1, x_2, x_3]$ given below:
\begin{align*}
f=&-27 x_{1}^{3}+27 x_{2}^{2} x_{3}-9 x_{3},\\
g=&-12 x_{1}^{3}+12 x_{1}^{2} x_{2}-12 x_{1}^{2} x_{3}+6 x_{1} x_{2}^{2}+36
x_{1} x_{2} x_{3} -33 x_{1} x_{3}^{2} \\ &+9 x_{2}^{3}-6 x_{2}^{2} x_{3}+6 x_{2}
x_{3}^{2}-6 x_{3}^{3}+3 x_{1}-6 x_{2}-6 x_{3}.
\end{align*}
These polynomials are orthogonally equivalent, i.e.~$f\sim_{O(n)}g$. A certificate $R\in O(n)$ verifying this equivalence is given by 

\begin{center}
$R\;=$
{\Large
$\left[
\renewcommand{\arraystretch}{1.3}
\begin{array}
[c]{rrr}%
\frac{2}{3} & -\frac{1}{3} & \frac{2}{3}\\
\frac{2}{3} & \frac{2}{3} & -\frac{1}{3}\\
-\frac{1}{3} & \frac{2}{3} & \frac{2}{3}
\end{array}
\right].$}
\end{center}
We may easily check that $R$ is an element of $O(n)$ and that $f(Rx)=g(x)$. 
Our goal in this paper is to develop an efficient and scalable algorithm, given in Algorithm \ref{alg:certificate} below, to compute the matrix $R$ above given input polynomials $f$ and $g$. 
\end{example}
\begin{remark}
In the above example (Example~\ref{ex:running}), the input polynomial coefficients and the output matrix entries lie in the same field, namely $\mathbb{Q}$. This is not always true.
For instance consider the following input.
\begin{align*}
 f &= 4 x_1^2 - 2 x_2^2 \\
 g &=   x_1^2 - 6 x_1 x_2 + x_2^2 
\end{align*} 
Then an output is given by 
$$
\left[
\begin{array}
[c]{rr}%
1/\sqrt{2} &  -1/\sqrt{2} \\
1/\sqrt{2} &  1/\sqrt{2}\\
\end{array}
\right] 
\;\;=\;\;
\left[
\begin{array}
[c]{rr}%
\cos(\pi/4) & -\sin(\pi/4) \\
\sin(\pi/4) & \cos(\pi/4)\\
\end{array}
\right],
$$
which corresponds to the rotation by $\pi/4$. Note that the input polynomial coefficients lie in $\mathbb{Q}$ but the output matrix entries lie in $\mathbb{Q}(\sqrt{2})$.
In general, the output matrix entries lie in an algebraic extension of the input coefficient field. Thus, any method that strives to solve the main problem (Problem~\ref{prb:problem}) exactly would encounter computations in and algebraic extension field, which severely limits the practicality of the method especially for large $n$ and $d$. Hence, one instead strives to develop methods that satisfy both of the following two requirements:
(1) the method can be executed exactly if needed  and
(2) the method can be safely executed approximately, i.e.,  the method is numerically stable.
\end{remark}

\begin{remark}[Relationship to symmetric tensors and the Procrustes problem]\label{remark:Procrustes}
    Note that \Cref{prb:problem} can also be thought of as a sort of orthogonal Procrustes problem on symmetric tensors. To see this we represent the degree-$d$ polynomials $f,g\in\mathbb{R}[x_{1},\ldots,x_{n}]$
    in terms of symmetric coefficient tensors as
    \begin{align*}
        f(x) &= F_0 + \langle F_1, x \rangle + \langle F_2, x \circ x \rangle + \cdots + \langle F_d, x^{\circ d} \rangle \\
        g(x) &= G_0 + \langle G_1, x \rangle + \langle G_2, x \circ x \rangle + \cdots + \langle G_d, x^{\circ d} \rangle
        ,
    \end{align*}
    where the scalars $F_0, G_0 \in \mathbb{R}$ represent the constant terms for $f$ and $g$, respectively,
    the vectors $F_1, G_1 \in \mathbb{R}^n$ represent their linear terms,
    the symmetric matrices $F_2, G_2 \in \mathbb{R}^{n \times n}$ represent their quadratic terms,
    and so on,
    $\circ$ denotes the usual tensor outer product,
    and $\langle \cdot, \cdot \rangle$ denotes the usual tensor inner product.
    It then follows by standard tensor algebra manipulations
    (see, e.g., \cite[Chapter 3]{ballard2025tdd}) that for any $R \in \mathbb{R}^{n \times n}$,
    we have
    \begin{align*}
        f(Rx)
        &
        =
        F_0 + \langle F_1, Rx \rangle + \langle F_2, (Rx) \circ (Rx) \rangle + \cdots + \langle F_d, (Rx)^{\circ d} \rangle
        \\&
        =
        F_0 + \langle F_1 \times_1 R^\top, x \rangle + \langle F_2 \times_1 R^\top \times_2 R^\top, x \circ x \rangle + \cdots + \langle F_d \times_1 R^\top \cdots \times_d R^\top, x^{\circ d} \rangle
        ,
    \end{align*}
    where $\times_i$ denotes the tensor-times-matrix product along mode $i$.
    Thus, we seek $R \in O(n)$ for which
    \begin{equation*}
        G_i = F_i \times_1 R^\top \cdots \times_i R^\top
        ,
        \quad
        \text{for each }
        i \in \{1,\dots,d\}
        ,
    \end{equation*}
    which can be formulated as the following least-squares optimization problem
    \begin{equation}
        \label{eq:procrustes:analogue}
        \operatorname*{argmin}_{R \in O(n)}
        \sum_{i=1}^d
        \| F_i \times_1 R^\top \cdots \times_i R^\top - G_i \|_{\mathsf{F}}^2
        ,
    \end{equation}
    where $\|\cdot\|_{\mathsf{F}}$ denotes the usual Frobenius norm for tensors.
    
    This closely resembles the well-studied orthogonal Procrustes problem \cite{schonemann1966gso},
    which seeks to find $R \in O(n)$ that minimizes $\| R A - B \|_{\mathsf{F}}^2$
    for given matrices $A$ and $B$.
    However, there are crucial differences.
    First off, the problem \cref{eq:procrustes:analogue} involves symmetric tensors of all orders up to $d$,
    rather than just matrices (i.e., order-2 tensors).
    Moreover, the orthogonal matrix $R$ is multiplied along all the modes of $F_i$.
    This means that even for the order-2 term in \cref{eq:procrustes:analogue},
    which seeks to transform the matrix $F_2 \in \mathbb{R}^{n \times n}$ into the matrix $G_2 \in \mathbb{R}^{n \times n}$,
    the transformation corresponds to $F_2 \times_1 R^\top \times_2 R^\top = R^\top F_2 R$
    rather than $R^\top F_2$ as in the orthogonal Procrustes problem.
    Hence, the usual method for solving the orthogonal Procrustes problem
    via polar factorization of $BA^\top$
    is not straightforward to extend to solve \cref{eq:procrustes:analogue}.
    In this paper, we develop a different approach that instead exploits intrinsic properties of the geometry of the polynomials.
\end{remark}

\section{Polynomial-Weighted Principal Component Analysis}
\label{sec:pwpca}

Taking inspiration from statistical multivariate analysis,
this section develops a polynomial-weighted form of principal component analysis
that we will use to tackle the problem of certifying orthogonal equivalence of polynomials stated in the previous section, Problem~\ref{prb:problem}.

\subsection{A brief review of Principal Component Analysis}
\label{sec:pca}

We begin with a brief review of principal component analysis (PCA)
\cite{hotelling1933aoa,pearson1901lol}
to set the stage for our proposed polynomial-weighted principal component analysis;
see, e.g., the monograph \cite{jolliffe2002pca} for a comprehensive survey.
Given a collection of $N$ data points $x^{(1)}, \dots, x^{(N)} \in \mathbb{R}^n$,
PCA obtains a set of \emph{principal axes} that maximally capture the variability
present in the data.
Since many datasets have underlying low-dimensional structure,
much of the meaningful variability
can often be captured with relatively few principal axes.
These principal axes are obtained
by finding directions
along which the variance of the data points is maximized.
Precisely put,
the first principal axis is obtained by maximizing%
\footnote{For simplicity,
we assume that the data points have already been centered
to have an average of $(x^{(1)} + \cdots + x^{(N)})/N = 0$;
this can always be straightforwardly accomplished
by subtracting the average from all the data points.}
\begin{equation}
    \label{eq:pca:dirvar}
    \operatorname{Var}_{u} (x^{(1)}, \dots, x^{(N)})
    =
    \frac{1}{N}
    \sum_{i=1}^N
    (u^\top x^{(i)})^2
    ,
\end{equation}
with respect to the unit direction vector $u \in \mathbb{R}^n$.
The second principal axis is then obtained
by maximizing $\operatorname{Var}_{u} (x^{(1)}, \dots, x^{(N)})$
with respect to unit direction vectors $u \in \mathbb{R}^n$
that are orthogonal to the first axis,
and so on.
This yields an ordered set of orthonormal principal axes $v_1, \dots, v_n \in \mathbb{R}^n$ as follows:
\begin{equation}
    \label{eq:pca:principalaxes}
    \begin{split}
    v_1
    &\in
    \operatorname*{argmax}_{u \in \mathbb{R}^{n} : \|u\| = 1}
    \operatorname{Var}_{u} (x^{(1)}, \dots, x^{(N)})
    , \\
    v_2
    &\in
    \operatorname*{argmax}_{u \in \mathbb{R}^{n} : \|u\| = 1}
    \operatorname{Var}_{u} (x^{(1)}, \dots, x^{(N)})
    \quad \text{s.t.} \quad
    u \perp v_1
    , \\
    &\;\;\vdots\\
    v_n
    &\in
    \operatorname*{argmax}_{u \in \mathbb{R}^{n} : \|u\| = 1}
    \operatorname{Var}_{u} (x^{(1)}, \dots, x^{(N)})
    \quad \text{s.t.} \quad
    u \perp \{v_1, \dots, v_{n-1}\}
    .
    \end{split}
\end{equation}
Notably,
the variance $\operatorname{Var}_{u} (x^{(1)}, \dots, x^{(N)})$ along the direction $u \in \mathbb{R}^n$
can be expressed as a quadratic form
\begin{equation*}
    \operatorname{Var}_{u} (x^{(1)}, \dots, x^{(N)})
    =
    \frac{1}{N}
    \sum_{i=1}^N
    (u^\top x^{(i)})
    ({x^{(i)}}^\top u)
    =
    u^\top
    \left(
    \frac{1}{N} \sum_{i=1}^N
    x^{(i)} {x^{(i)}}^\top
    \right)
    u
    =
    u^\top
    \operatorname{Cov}(x^{(1)}, \dots, x^{(N)})
    u
    ,
\end{equation*}
defined by the covariance matrix
\begin{equation}
    \label{eq:pca:cov}
    \operatorname{Cov}(x^{(1)}, \dots, x^{(N)})
    =
    \frac{1}{N} \sum_{i=1}^N
    x^{(i)} {x^{(i)}}^\top
    \in \mathbb{R}^{n \times n}
    .
\end{equation}
As a result,
the solutions to the sequence of optimization problems in \cref{eq:pca:principalaxes}
are given by the eigenvectors of the covariance matrix,
i.e.,
$v_1,\dots,v_n$ are the eigenvectors of $\operatorname{Cov}(x^{(1)}, \dots, x^{(N)})$
in order of decreasing eigenvalue $\lambda_1 \geq \cdots \geq \lambda_n$.
Consequently, the principal axes are also sometimes equivalently defined
as the eigenvectors of the covariance matrix.
It should be noted here that the principal axes, like eigenvectors,
are only unique when the eigenvalues are distinct
(i.e., when $\lambda_1 > \cdots > \lambda_n$)
and even then are only unique up to sign.
Finally,
also note that while PCA is often first presented for finite collections of data points,
it immediately generalizes to data distributed according to arbitrary probability measures $\mu$
by replacing the above finite averages with integrals over the measure.
Namely, the above variance and covariance become
\begin{align}
    \operatorname{Var}_{u} (\mu)
    &
    =
    \int
    (u^\top x)^2
    \; d\mu(x)
    , &
    \operatorname{Cov}(\mu)
    &
    =
    \int
    x x^\top
    \; d\mu(x)
    ,
\end{align}
which are often written concisely as the expectations $\mathbb{E}[(u^\top x)^2]$ and $\mathbb{E}[x x^\top]$, respectively.

\subsection{Definition of a polynomial-weighted variance and covariance}

We now define a polynomial-weighted notion of variance and covariance
for polynomials $f \in \mathbb{R}[x_1,\dots,x_n]$
that will form the basis for the proposed polynomial-weighted PCA.
Similar to how \cref{eq:pca:dirvar,eq:pca:cov} capture the ``spread'' of data in different directions,
the overall idea here is to capture how ``large'' $f$ is along any given direction $u \in \mathbb{R}^n$.
Focusing on homogeneous polynomials $f$,
we restrict our study as usual to the unit sphere
\begin{equation}
    \mathbb{S}^{n-1} = \{ x \in \mathbb{R}^n : x_1^2 + \cdots + x_n^2 = 1 \}
    ,
\end{equation}
without loss of generality.
This leads naturally to the following polynomial-weighted notion of variance.

\begin{definition}[Polynomial-weighted variance of a homogeneous polynomial]
    \label{def:pwvar}
    The polynomial-weighted variance of a homogeneous polynomial $f \in \mathbb{R}[x_1,\dots,x_n]$
    along a direction $u \in \mathbb{R}^n$
    is defined as
    \begin{equation}
        \label{eq:pwpca:dirvar}
        \operatorname{Var}_{u} (f)
        =
        \int_{\mathbb{S}^{n-1}}
        f(x)^2
        \,
        (u^\top x)^2
        \; d\mu(x)
        ,
    \end{equation}
    where $\mu$ denotes the spherical measure on $\mathbb{S}^{n-1}$.
\end{definition}

In words, $\operatorname{Var}_{u} (f)$
is the variance along the direction $u$
for points $x$ uniformly distributed on the sphere $\mathbb{S}^{n-1}$,
weighted by their corresponding (squared) polynomial values $f(x)^2$.
As we show below in \cref{thm:equiv:pwvarcov}, this notion captures the ``orientation'' of $f$.
Weighted PCA variants using weighted notions of variance
have also been used in statistical contexts, e.g.,
to account for the relative noise level of different points \cite{hong2023owp}.
Here, the polynomial $f$ sets the amount of weight given to each point.

As with the variance of data along different directions defined in \cref{eq:pca:dirvar},
the polynomial-weighted variance $\operatorname{Var}_{u} (f)$
can be expressed as a quadratic form as follows
\begin{equation*}
    \operatorname{Var}_{u} (f)
    =
    \int_{\mathbb{S}^{n-1}}
    f(x)^2
    \,
    (u^\top x)
    (x^\top u)
    \; d\mu(x)
    =
    u^\top
    \left(
    \int_{\mathbb{S}^{n-1}}
    f(x)^2
    \,
    x x^\top
    \; d\mu(x)
    \right)
    u
    .
\end{equation*}
This leads to the following polynomial-weighted notion of covariance matrix.

\begin{definition}[Polynomial-weighted covariance of a homogeneous polynomial]
    \label{def:pwcov}
    The polynomial-weighted covariance matrix of a homogeneous polynomial $f \in \mathbb{R}[x_1,\dots,x_n]$
    is defined as
    \begin{equation}
        \label{eq:pwpca:cov}
        \operatorname{Cov}(f)
        =
        \int_{\mathbb{S}^{n-1}}
        f(x)^2
        \,
        x x^\top
        \; d\mu(x)
        \in \mathbb{R}^{n \times n}
        ,
    \end{equation}
    where $\mu$ denotes the spherical measure on $\mathbb{S}^{n-1}$.
    Under this definition,
    we have $\operatorname{Var}_{u} (f) = u^\top \operatorname{Cov}(f) u$.
\end{definition}

To gain some further intuition for these notions,
note that the polynomial-weighted variance and covariance
can be written equivalently as
\begin{align*}
    \operatorname{Var}_{u} (f)
    &
    =
    \rho_f
    \cdot
    \int_{\mathbb{S}^{n-1}}
    (u^\top x)^2
    \;
    \left(
    \frac{f(x)^2}{\rho_f}
    \right)
    d\mu(x)
    ,
    &
    \operatorname{Cov}(f)
    &
    =
    \rho_f
    \cdot
    \int_{\mathbb{S}^{n-1}}
    x x^\top
    \;
    \left(
    \frac{f(x)^2}{\rho_f}
    \right)
    d\mu(x)
    ,
\end{align*}
where
$\rho_f = \int_{\mathbb{S}^{n-1}} f(x)^2 \; d\mu(x)$
normalizes $f(x)^2$ to be a probability density
(i.e., $f(x)^2/\rho_f$ integrates to one).
In this interpretation,
the polynomial-weighted variance and covariance
can be thought of as the variance and covariance
for points distributed on the sphere $\mathbb{S}^{n-1}$
with probability density given by $f(x)^2/\rho_f$.

\subsection{Main definition: Polynomial-Weighted Principal Component Analysis}

We are now ready to describe the proposed polynomial-weighted principal component analysis (PW-PCA).
The overall idea is to define principal axes for polynomials
analogous to those defined in \cref{eq:pca:principalaxes} for data
by now maximizing the polynomial-weighted variance (\cref{def:pwvar}).
Since the polynomial-weighted variance is only defined for homogeneous polynomials,
we work with the homogenization
$\overline{f} \in \mathbb{R}[x_1,\dots,x_{n+1}]$
defined as
\begin{equation}
    \label{eq:homo}
    \overline{f}(x_1,\dots,x_{n+1})
    =
    x_{n+1}^{\deg(f)} f(x_1/x_{n+1},\dots,x_n/x_{n+1})
    .
\end{equation}
This leads to the following definition for PW-PCA.
It is mostly a straightforward analogue of \cref{eq:pca:principalaxes},
with some small subtleties in properly handling the homogenizing variable $x_{n+1}$.

\begin{definition}[Polynomial-Weighted Principal Component Analysis]
    \label{def:pwpca}
    For a polynomial $f \in \mathbb{R}[x_1,\dots,x_n]$,
    we call $v_1,\dots,v_n \in \mathbb{R}^n$ a set of principal axes for $f$
    if
    \begin{equation}
        \label{eq:pwpca:principalaxes}
        \begin{split}
        v_1
        &\in
        \operatorname*{argmax}_{u \in \mathbb{R}^{n} : \|u\| = 1}
        \operatorname{Var}_{[u,0]} (\overline{f})
        , \\
        v_2
        &\in
        \operatorname*{argmax}_{u \in \mathbb{R}^{n} : \|u\| = 1}
        \operatorname{Var}_{[u,0]} (\overline{f})
        \quad \text{s.t.} \quad
        u \perp v_1
        , \\
        &\;\;\vdots\\
        v_n
        &\in
        \operatorname*{argmax}_{u \in \mathbb{R}^{n} : \|u\| = 1}
        \operatorname{Var}_{[u,0]} (\overline{f})
        \quad \text{s.t.} \quad
        u \perp \{v_1, \dots, v_{n-1}\}
        .
        \end{split}
    \end{equation}
    We call the associated variances
    $\lambda_i = \operatorname{Var}_{[v_i,0]} (\overline{f})$
    the principal variances,
    and call the pair $(\lambda, V)$ a PW-PCA of $f$,
    where
    $\lambda = (\lambda_1,\dots,\lambda_n) \in \mathbb{R}^n$
    and
    $V = (v_1,\dots,v_n) \in \mathbb{R}^{n \times n}$.
\end{definition}

In words,
the principal axes $v_1,\dots,v_n$ provide an orthonormal basis for $\mathbb{R}^n$
that captures the directions along which the homogenization $\overline{f}$ is large,
and the principal variances $\lambda_1,\dots,\lambda_n$ provide the corresponding
polynomial-weighted variances.

\subsection{Algorithm for Polynomial-Weighted Principal Component Analysis}

We now develop some theory that will lead to an efficient algorithm for computing
the Polynomial-Weighted Principal Component Analysis (PW-PCA)
introduced in \cref{def:pwpca}.
The first theoretical result reduces
the sequence of optimization problems in \cref{eq:pwpca:principalaxes}
to eigendecomposition of the leading principal $n \times n$ submatrix
of polynomial-weighted covariance matrix (\cref{def:pwcov}).

\begin{proposition}[PW-PCA via eigendecomposition]
    \label{thm:pwpca:cov}
    For any $f \in \mathbb{R}[x_1,\dots,x_n]$,
    the following are equivalent:
    \begin{enumerate}
        \item $(\lambda, V)$ is a PW-PCA of $f$
        \item $V \operatorname{diag}(\lambda) V^\top$ is an orthogonal eigendecomposition of $\operatorname{Cov}(\overline{f})_{1:n, 1:n}$
    \end{enumerate}
    where the eigendecomposition here sorts the eigenvalues in non-increasing order,
    i.e., $\lambda_1 \geq \cdots \geq \lambda_n$.
\end{proposition}

This \lcnamecref{thm:pwpca:cov} is analogous to the fact
that the principal axes for data points
are given by eigenvectors of the covariance matrix
(as described in \cref{sec:pca}),
and it can be proven in essentially the same manner.
A typical approach is to characterize the solutions
to the constrained optimization problems in \cref{eq:pwpca:principalaxes}
via Lagrange multipliers,
e.g., as in \cite[Section~1.1]{jolliffe2002pca}.
Here we provide an alternative elementary proof that uses another common approach
(writing the covariance in terms of its eigenvalues and eigenvectors),
which can sometimes be more enlightening.

\begin{proof}[Proof of \cref{thm:pwpca:cov}]
    Using the relationship described in \cref{def:pwcov}
    between the polynomial-weighted variance and covariance,
    we first have that
    \begin{equation*}
        \operatorname{Var}_{[u,0]} (\overline{f})
        = [u, 0]^\top \operatorname{Cov}(\overline{f}) [u, 0]
        =
        \begin{bmatrix} u \\ 0 \end{bmatrix}^\top
        \begin{bmatrix}
            \operatorname{Cov}(\overline{f})_{1:n, 1:n} & \operatorname{Cov}(\overline{f})_{1:n, n+1} \\
            \operatorname{Cov}(\overline{f})_{n+1, 1:n} & \operatorname{Cov}(\overline{f})_{n+1, n+1}
        \end{bmatrix}
        \begin{bmatrix} u \\ 0 \end{bmatrix}
        =
        u^\top \operatorname{Cov}(\overline{f})_{1:n, 1:n} u
        .
    \end{equation*}
    Since $\operatorname{Var}_{[u,0]} (\overline{f}) \geq 0$ for any $u \in \mathbb{R}^n$,
    it follows that $\operatorname{Cov}(\overline{f})_{1:n, 1:n}$ is positive semidefinite
    and its eigenvalues are all nonnegative.
    We now prove each direction one at a time:
    \begin{itemize}
        \item $(2) \implies (1)$.

        Suppose $V \operatorname{diag}(\lambda) V^\top$
        is an orthogonal eigendecomposition
        of $\operatorname{Cov}(\overline{f})_{1:n, 1:n}$,
        i.e., $\lambda_1 \geq \cdots \geq \lambda_n \geq 0$
        are its eigenvalues
        with corresponding orthonormal eigenvectors $v_1,\dots,v_n$
        given by the columns of $V$.
        Substituting then yields
        \begin{equation*}
            \operatorname{Var}_{[u,0]} (\overline{f})
            = u^\top \operatorname{Cov}(\overline{f})_{1:n, 1:n} u
            = u^\top V \operatorname{diag}(\lambda) V^\top u
            = \lambda_1 (v_1^\top u)^2 + \cdots + \lambda_n (v_n^\top u)^2
            .
        \end{equation*}
        Thus, for any $u$ with $\|u\| = 1$, we have the upper bound
        \begin{equation*}
            \operatorname{Var}_{[u,0]} (\overline{f})
            \leq \lambda_1 (v_1^\top u)^2 + \cdots + \lambda_1 (v_n^\top u)^2
            = \lambda_1 \| V^\top u \|^2
            = \lambda_1 \| u \|^2
            = \lambda_1
            ,
        \end{equation*}
        which is achieved by $v_1$
        since
        $\operatorname{Var}_{[v_1,0]} (\overline{f}) = \lambda_1 (v_1^\top v_1)^2 + \cdots + \lambda_n (v_n^\top v_1)^2 = \lambda_1$,
        and so
        \begin{equation*}
            v_1
            \in
            \operatorname*{argmax}_{u \in \mathbb{R}^{n} : \|u\| = 1}
            \operatorname{Var}_{[u,0]} (\overline{f})
            .
        \end{equation*}
        Similarly, for any $u$ with $\|u\| = 1$ such that $u \perp v_1$, we have the upper bound
        \begin{equation*}
            \operatorname{Var}_{[u,0]} (\overline{f})
            = \lambda_2 (v_2^\top u)^2 + \cdots + \lambda_n (v_n^\top u)^2
            \leq \lambda_2 (v_2^\top u)^2 + \cdots + \lambda_2 (v_n^\top u)^2
            = \lambda_2 \| V^\top u \|^2
            = \lambda_2 \| u \|^2
            = \lambda_2
            ,
        \end{equation*}
        which is achieved by $v_2$
        since
        $\operatorname{Var}_{[v_2,0]} (\overline{f}) = \lambda_1 (v_1^\top v_2)^2 + \cdots + \lambda_n (v_n^\top v_2)^2 = \lambda_2$,
        and so
        \begin{equation*}
            v_2
            \in
            \operatorname*{argmax}_{u \in \mathbb{R}^{n} : \|u\| = 1}
            \operatorname{Var}_{[u,0]} (\overline{f})
            \quad \text{s.t.} \quad
            u \perp v_1
            .
        \end{equation*}
        It follows in a similar fashion that $v_3,\dots,v_n$
        solve the associated optimization problems in \cref{eq:pwpca:principalaxes},
        and hence $v_1,\dots,v_n$ form a set of principal axes for $f$
        with associated principal variances given by $\lambda_1,\dots,\lambda_n$.
        This completes the proof of this direction.

        \item $(1) \implies (2)$.

        Suppose $(\lambda, V)$ is a PW-PCA of $f$,
        i.e., the columns $v_1,\dots,v_n$ of $V$ are a set of principal axes for $f$
        with associated principal variances $\lambda_1 \geq \cdots \geq \lambda_n \geq 0$.
        Next, recall that $\operatorname{Cov}(\overline{f})_{1:n, 1:n}$ is a symmetric matrix,
        so it always has
        an orthogonal eigendecomposition $\tilde{V} \operatorname{diag}(\tilde{\lambda}) \tilde{V}^\top$.
        Then by the same arguments as above,
        for any $u$ with $\|u\| = 1$,
        we have the upper bound
        \begin{equation*}
            \operatorname{Var}_{[u,0]} (\overline{f})
            = u^\top \operatorname{Cov}(\overline{f})_{1:n, 1:n} u
            = u^\top \tilde{V} \operatorname{diag}(\tilde{\lambda}) \tilde{V}^\top u
            = \tilde{\lambda}_1 (\tilde{v}_1^\top u)^2 + \cdots + \tilde{\lambda}_n (\tilde{v}_n^\top u)^2
            \leq \tilde{\lambda}_1
            ,
        \end{equation*}
        which is achieved by $\tilde{v}_1$.
        Note now that this upper-bound is also achieved by any unit vector
        in the eigenspace associated with $\tilde{\lambda}_1$
        (i.e., the span of the eigenvectors with equal eigenvalue),
        and is moreover not achieved by any other unit vector.
        Thus, we have that
        \begin{equation*}
            \operatorname{span}(\tilde{v}_i : \tilde{\lambda}_i = \tilde{\lambda}_1)
            \cap
            \mathbb{S}^{n-1}
            =
            \operatorname*{argmax}_{u \in \mathbb{R}^{n} : \|u\| = 1}
            \operatorname{Var}_{[u,0]} (\overline{f})
            ,
        \end{equation*}
        and so it follows that $v_1 \in \operatorname{span}(\tilde{v}_i : \tilde{\lambda}_i = \tilde{\lambda}_1)$
        and $\lambda_1 = \tilde{\lambda}_1$.
        Hence, $v_1$ is an eigenvector of $\operatorname{Cov}(\overline{f})_{1:n, 1:n}$
        with corresponding eigenvalue $\lambda_1$.

        By similar arguments,
        it follows that $v_2,\dots,v_n$ are all eigenvectors of $\operatorname{Cov}(\overline{f})_{1:n, 1:n}$
        with corresponding eigenvalues $\lambda_2 \geq \cdots \geq \lambda_n$.
        In other words,
        $\operatorname{Cov}(\overline{f})_{1:n, 1:n} v_i = \lambda_i v_i$ for each $i \in \{1,\dots,n\}$,
        which can be written concisely in matrix form as
        \begin{equation*}
            \operatorname{Cov}(\overline{f})_{1:n, 1:n} V
            =
            [\operatorname{Cov}(\overline{f})_{1:n, 1:n} v_1, \dots, \operatorname{Cov}(\overline{f})_{1:n, 1:n} v_n]
            =
            [\lambda_1 v_1, \dots, \lambda_n v_n]
            =
            V \operatorname{diag}(\lambda)
            .
        \end{equation*}
        Now, recall that by construction $v_1,\dots,v_n$ are orthonormal
        so $V$ is an orthogonal matrix and hence
        \begin{equation*}
            V \operatorname{diag}(\lambda) V^\top
            =
            \operatorname{Cov}(\overline{f})_{1:n, 1:n} V V^\top
            =
            \operatorname{Cov}(\overline{f})_{1:n, 1:n}
            ,
        \end{equation*}
        i.e., $V \operatorname{diag}(\lambda) V^\top$ is an orthogonal eigendecomposition
        of $\operatorname{Cov}(\overline{f})_{1:n, 1:n}$.
        This completes the proof of this direction.
    \end{itemize}
\end{proof}

It now remains to have an efficient way of computing
the polynomial-weighted covariance matrix.
Notably, its definition involves multivariate surface integrations
which are not straightforward to implement directly.
Here, we provide an explicit algebraic expression that can be implemented easily and efficiently.
For this, we first recall and introduce some standard notations.

\begin{notation}
We use the following notations:

\begin{itemize}

\item $\Gamma$ denotes the gamma function.

\item $!!$ denotes the double factorial. 

\item $e_{i}$ denotes the $i$-th standard basis vector in $\mathbb{R}^n$.

\item $\eta(s) =  \begin{cases}
s, & \text{if $s$ is odd} , \\
s-1, & \text{if $s$ is even} , \\
\end{cases}
$ ``rounds'' $s$ down to the nearest odd number.

\end{itemize}
\end{notation}

We now state the algebraic expression
for the polynomial-weighted covariance
that we will use to obtain an efficient algorithm
for computing PW-PCA's.
It shows that each term of the polynomial
only contributes to a few entries of the polynomial-weighted covariance.

\begin{theorem}
[Expression for the polynomial-weighted covariance] \label{thm:pwcov:formulas}
Let $f\in\mathbb{R}[x_{1},\dots,x_{n}]$ be a homogeneous polynomial of degree
$d$.
Let $a \cdot x^{\mu}$ denote each term in $f^2$, that is,
$f^2 = \sum{a \cdot x^{\mu}}$.
Then we have
\begin{equation}
\operatorname{Cov}(f)\;\;=\;\;\frac{\pi^{n/2}}{2^{d} \;\Gamma(n/2+d+1)}
\cdot \sum\,a\cdot\Psi(x^{\mu})
\end{equation}
where $\Psi(x^{\mu})\in\mathbb{N}_{\geq0}^{n\times n}$ is given by
\begin{equation}
\Psi(x^{\mu})=\left(  \prod_{k=1}^{n}\eta(\mu_{k})!!\right)
\begin{cases}
\operatorname{diag}(\mu)+I_{n}, & \text{if $\mu$ has no odd entries},\\
e_{u}e_{v}^{\top}+e_{v}e_{u}^{\top}, & \text{if $\mu$ has two odd entries at
indices $u$ and $v$},\\
0_{n\times n}, & \text{otherwise}.
\end{cases}
\end{equation}

\end{theorem}

The proof of \cref{thm:pwcov:formulas} combines standard formulas \cite{Folland2001}
for integrating monomials over the sphere
(i.e., for computing statistical moments for the uniform distribution on the sphere)
with a careful analysis of the monomial exponents that appear in the polynomial-weighted covariance.
In particular,
our analysis uncovers and exploits sparsity in the contribution of each polynomial term
to the polynomial-weighted covariance.

\begin{proof}
[Proof of \cref{thm:pwcov:formulas}] Applying
\cref{eq:pwpca:cov} in Definition~\ref{def:pwcov} and factoring yields
\begin{equation}
\operatorname{Cov}(f)
=\int_{\mathbb{S}^{n-1}}\left(  \sum a\cdot x^{\mu}\right)  \,xx^{\top}\;d\mu(x)
=\sum a\cdot\int_{\mathbb{S}^{n-1}}x^{\mu}\,xx^{\top}\;d\mu(x)
=\sum a\cdot\operatorname{Cov}(x^{\mu}).\label{eq:cov:mons}%
\end{equation}
Hence, the problem reduces to computing covariance matrices for monomials $x^{\mu
}$ of degree $2d$. That is, we need to compute $\operatorname{Cov}(x^{\mu})$,
whose entries are given by
\[
\operatorname{Cov}(x^{\mu})_{ij}=\left(  \int_{\mathbb{S}^{n-1}}x^{\mu
}\,xx^{\top}\;d\mu(x)\right)  _{ij}=\int_{\mathbb{S}^{n-1}}x^{\mu}\,x_{i}%
x_{j}\;d\mu(x)=\int_{\mathbb{S}^{n-1}}x^{\mu+e_{i}+e_{j}}\;d\mu(x).
\]
Formulas for integrating these monomials over the unit sphere are provided in
\cite{Folland2001}. Specifically, applying the main theorem in
\cite{Folland2001}, we obtain the following two cases
\begin{equation}%
\begin{split}
\text{if $\mu+e_{i}+e_{j}$ has no odd entries} &  :\quad\operatorname{Cov}%
(x^{\mu})_{ij}=2\frac{\prod_{k=1}^{n}\Gamma(1/2+(\mu+e_{i}+e_{j})_{k}%
/2)}{\Gamma(\sum_{k=1}^{n}(1/2+(\mu+e_{i}+e_{j})_{k}/2))},\\
\text{otherwise} &  :\quad\operatorname{Cov}(x^{\mu})_{ij}=0.
\end{split}
\label{eq:cov:mon:entry}%
\end{equation}
Since the sum of the entries of $\mu+e_{i}+e_{j}$ is always $2d+2$, the
denominator in the first case simplifies to
\begin{equation}
\Gamma\left(  \sum_{k=1}^{n}(1/2+(\mu+e_{i}+e_{j})_{k}/2)\right)
=\Gamma(n/2+(2d+2)/2)=\Gamma(n/2+d+1).\label{eq:cov:mon:entry:nonzero:denom}%
\end{equation}
Moreover, all the entries of $(\mu+e_{i}+e_{j})/2$ are integers for this
\textquotedblleft only even entries\textquotedblright\ case, so the numerator
can also be rewritten as
\begin{align}
&  \prod_{k=1}^{n}\Gamma(1/2+(\mu+e_{i}+e_{j})_{k}/2)=\prod_{k=1}^{n}\left(
\frac{(2(\mu+e_{i}+e_{j})_{k}/2-1)!!}{2^{(\mu+e_{i}+e_{j})_{k}/2}}\sqrt{\pi
}\right)  \label{eq:cov:mon:entry:nonzero:numer}\\
&  \qquad\qquad=\pi^{n/2}\frac{\prod_{k=1}^{n}(2(\mu+e_{i}+e_{j})_{k}%
/2-1)!!}{2^{\sum_{k=1}^{n}(\mu+e_{i}+e_{j})_{k}/2}}=\frac{\pi^{n/2}}{2^{d+1}%
}\prod_{k=1}^{n}((\mu+e_{i}+e_{j})_{k}-1)!!.\nonumber
\end{align}
Substituting
\cref{eq:cov:mon:entry:nonzero:denom,eq:cov:mon:entry:nonzero:numer} back into
\cref{eq:cov:mon:entry} and factoring yields
\begin{equation}
\operatorname{Cov}(x^{\mu})=\frac{\pi^{n/2}}{2^{d}\;\Gamma
(n/2+d+1)}\cdot\Psi(x^{\mu}),\label{eq:mon:psi}%
\end{equation}
where the entries of $\Psi(x^{\mu})\in\mathbb{N}_{\geq0}^{n\times n}$ are defined
as
\begin{equation}
\Psi(x^{\mu})_{ij}=%
\begin{cases}
\prod_{k=1}^{n}((\mu+e_{i}+e_{j})_{k}-1)!!, & \text{if $\mu+e_{i}+e_{j}$ has
no odd entries},\\
0, & \text{otherwise}.
\end{cases}
\label{eq:psi:even}%
\end{equation}
Note now that there are only two scenarios in which $\mu+e_{i}+e_{j}$ ends up
having no odd entries: 

\begin{enumerate}

\item $\mu$ has no odd entries. In this scenario, $\mu+e_{i}+e_{j}$ has no odd
entries when $i=j$, so $\Psi(x^{\mu})$ is a diagonal matrix with entries
\[
\Psi(x^{\mu})_{ii}=\prod_{k=1}^{n}((\mu+e_{i}+e_{i})_{k}-1)!!=\underbrace{((\mu
_{i}+2)-1)!!}_{(\mu_{i}+1)(\mu_{i}-1)!!}\prod_{k\neq i}(\mu_{k}-1)!!=(\mu
_{i}+1)\prod_{k=1}^{n}(\mu_{k}-1)!!.
\]
Thus, we have for this case that
\begin{equation}
\Psi(x^{\mu})=\left(  \prod_{k=1}^{n}(\mu_{k}-1)!!\right)  \cdot
(\operatorname{diag}(\mu)+I_{n})=\left(  \prod_{k=1}^{n}\eta(\mu
_{k})!!\right)  \cdot(\operatorname{diag}(\mu)+I_{n}).\label{eq:psi:even1}%
\end{equation}

\item $\mu$ has two odd entries at indices $u$ and $v$. In this scenario,
$\mu+e_{i}+e_{j}$ has no odd entries when $(i,j)=(u,v)$ or $(i,j)=(v,u)$, so
$\Psi(x^{\mu})$ has only the following two nonzero entries
\[
\Psi(x^{\mu})_{uv}=\Psi(x^{\mu})_{vu}=\prod_{k=1}^{n}((\mu+e_{u}+e_{v})_{k}%
-1)!!=((\mu_{u}+1)-1)!!((\mu_{v}+1)-1)!!\prod_{k\neq u,v}(\mu_{k}-1)!!.
\]
Thus, we have for this case that
\begin{equation}
\Psi(x^{\mu})=((\mu_{u}+1)-1)!!((\mu_{v}+1)-1)!!\prod_{k\neq u,v}(\mu
_{k}-1)!!\cdot(e_{u}e_{v}^{\top}+e_{v}e_{u}^{\top})=\left(  \prod_{k=1}%
^{n}\eta(\mu_{k})!!)\right)  \cdot(e_{u}e_{v}^{\top}+e_{v}e_{u}^{\top
}).\label{eq:psi:even2}%
\end{equation}

\end{enumerate}
For other scenarios, $\mu+e_{i}+e_{j}$ always has at least one odd entry which gives us that
$\Psi(x^{\mu})=0_{n\times n}$ in these cases.
The proof concludes by combining the equations  \eqref{eq:cov:mons},
\eqref{eq:mon:psi},  \eqref{eq:psi:even},  \eqref{eq:psi:even1}, and
\eqref{eq:psi:even2}  and factoring.
\end{proof}

We now have a path to efficient computation of PW-PCA's.
In particular, it follows from \cref{thm:pwpca:cov}
that we need only compute an orthogonal eigendecomposition
for the leading principal $n \times n$ submatrix of the polynomial-weighted covariance matrix,
which can be computed efficiently via the expression in \cref{thm:pwcov:formulas}.
This leads to the following algorithm.

\begin{algorithm}
[\textsf{PW-PCA}]\ \label{alg:pwpca}\medskip

{\bf Input:} $f\in\mathbb{R}[x_{1},\ldots,x_{n}]$\medskip

{\bf Output:} $\left(  \lambda,V\right)  $, a PW-PCA of $f$

\begin{enumerate}
\item $\overline{f}\leftarrow x_{n+1}^{\deg f}f\left(  x_{1}/x_{n+1}%
,\ldots,x_{n}/x_{n+1}\right)  $, the homogenization of $f$

\item $\overline{C}\leftarrow\operatorname{Cov}(\overline{f})$, by applying
Theorem~\ref{thm:pwcov:formulas} to~$\overline{f}$

\item $C\leftarrow\overline{C}_{1:n,1:n}$

\item $\lambda\leftarrow$ the list of $n$ eigenvalues of $C$ in non-increasing order

\item $V\leftarrow$ the matrix whose columns are the corresponding
eigenvectors of $C$
\end{enumerate}
\end{algorithm}

\bigskip

We illustrate the above algorithm on the running example.   
We executed the  algorithm  both exactly (exact arithmetic) and approximately (double precision floating point arithmetic); in the approximate case we show only 3 decimal digits. 

\begin{example}[Running example continued]\mbox

Let $f=-27x_{1}^{3}+27x_{2}^{2}x_{3}-9x_{3}$ be the input to the PW-PCA algorithm (Algorithm \ref{alg:pwpca}). Below we carry out the steps of the algorithm. 

\begin{enumerate}
\item $\overline{f}\leftarrow-27 x_{1}^{3}+27 x_{2}^{2} x_{3}-9 x_{3}
x_{4}^{2}$.

\item $\overline{C}\leftarrow\frac{\pi^{2}}{2^{3}\cdot5!}\left[
\begin{array}
[c]{cccc}%
78489 & 0 & -2916 & 0\\
0 & 20655 & 0 & 0\\
-2916 & 0 & 16767 & 0\\
0 & 0 & 0 & 12879
\end{array}
\right]  \allowbreak$

$\hspace{11pt}\approx\left[
\begin{array}
[c]{rrrr}%
806.933 & 0.000 & -29.979 & 0.000\\
0.000 & 212.351 & 0.000 & 0.000\\
-29.979 & 0.000 & 172.379 & 0.000\\
0.000 & 0.000 & 0.000 & 132.407
\end{array}
\right]$

\item $C\leftarrow\frac{\pi^{2}}{2^{3}\cdot5!}\left[
\begin{array}
[c]{ccc}%
78489 & 0 & -2916\\
0 & 20655 & 0\\
-2916 & 0 & 16767
\end{array}
\right]  $

\hspace{11pt}$\approx\left[
\begin{array}
[c]{rrr}%
806.933 & 0.000 & -29.979\\
0.000 & 212.351 & 0.000\\
-29.979 & 0.000 & 172.379
\end{array}
\right]  $

\item $\lambda\leftarrow\frac{\pi^{2}}{2^{3}\cdot5!}\left[
\begin{array}
[c]{ccc}%
47628+243\sqrt{16273} & 20655 & 47628-243\sqrt{16273}%
\end{array}
\right]  $

\hspace{11pt} $\approx\left[  808.346,\;212.351,\;170.966\right]  $

\item $V\leftarrow\left[
\begin{array}
[c]{ccc}%
\frac{\sqrt{32546-254\sqrt{16273}}}{24}+\frac{127\sqrt{32546-254\sqrt{16273}%
}\,\sqrt{16273}}{390552} & 0 & \frac{\sqrt{32546+254\sqrt{16273}}}{24}%
-\frac{127\sqrt{32546+254\sqrt{16273}}\,\sqrt{16273}}{390552}\\
0 & 1 & 0\\
-\frac{\sqrt{32546-254\sqrt{16273}}\,\sqrt{16273}}{32546} & 0 & \frac
{\sqrt{32546+254\sqrt{16273}}\,\sqrt{16273}}{32546}%
\end{array}
\right]  $

\hspace{12pt} $\approx\left[
\begin{array}
[c]{rrr}%
0.999 & 0.000 & 0.047\\
0.000 & -1.000 & 0.000\\
-0.047 & 0.000 & 0.999
\end{array}
\right]  $
\end{enumerate}

The output of Algorithm \ref{alg:pwpca} is then $(\lambda, V)$, which is the PW-PCA given in approximate form by $$
(\lambda, V)\approx \left(\left[  808.346,\;212.351,\;170.966\right]  ,\left[
\begin{array}
[c]{rrr}%
0.999 & 0.000 & 0.047\\
0.000 & -1.000 & 0.000\\
-0.047 & 0.000 & 0.999
\end{array}
\right]\right).
$$ Note that to represent the exact version of $(\lambda, V)$ on a computer for algorithmic purposes we are required to work in some algebraic field extension of the rationals. 
\end{example}

\section{Certificate of Orthogonal Equivalence via PW-PCA}
\label{sec:coe_pwpca}

In this section, we will
develop the required mathematical theory, along with a resulting algorithm, to tackle the main problem
of this paper (Problem \ref{prb:problem}).   
To accomplish this we will  investigate the  connection between orthogonal equivalence and  polynomial-weighted principal
component analysis (PW-PCA) introduced in the previous section.
 The key results are contained
in Theorem~\ref{thm:ortho:equiv}. We will first state the key results
precisely and then we will provide a detailed proof via several lemmas.
Finally, we will recast the key results into an algorithm.

\medskip 

As is commonly done in the theory of PCA, the key results will be stated for polynomials  whose  principal axes are uniquely determined  up to sign (see Definition \ref{def:signflip(n)}), or equivalently, will be stated for polynomials whose
whose principal variances are distinct. As expected almost all polynomials have principal variances which are distinct; in other words a generic polynomial has distinct principal variances. Throughout this paper we will assume all polynomials satisfy this condition, stated precisely in Assumption \ref{cond:simple:eigvals} below.
\begin{assumption}
[Global assumption: distinct principal variances]\label{cond:simple:eigvals} From now on, we assume
that every polynomial satisfies the following condition: its principal variances $\lambda_1,\dots,\lambda_n$ are distinct, i.e., $\lambda_1 > \dots > \lambda_n$.
\end{assumption}

As observed above the property of having distinct principal variances is true for a Zariski dense set of polynomials, that is a generic polynomial will have this property. This is stated precisely in Proposition \ref{prp:genericity} below. 
\begin{proposition}
[Genericity of distinct principal variances]\label{prp:genericity} Let $n\geq2$ and
$d\geq2$. Assumption \ref{cond:simple:eigvals} is \emph{generic} in the sense that it holds for almost all
$f\in\mathbb{R}\left[  x_{1},\ldots,x_{n}\right]  $ of degree at most $d$.
\end{proposition}

\begin{proof}
The proof is conceptually straightforward and is based on a standard arguments which will show that the set of polynomials which do not have distinct principal variances is closed in the Zariski topology. However, since the proof  involves several long and tedious algebraic manipulations which do not add significant theoretical insight, we opt to defer the proof to Appendix \ref{appendix:A}.
\end{proof}

\smallskip

In order to deal with the intrinsic sign ambiguity precisely and systematically we   introduce   a  finite group, which we will refer to as the \textquotedblleft
signflip\textquotedblright group, and similar to orthogonal group (see
Definition~\ref{def:O(n)}) we give notations for the associated action
and equivalence.

\begin{definition}
[Signflip group, action and equivalence]\label{def:signflip(n)}\ \ 

\begin{enumerate}
\item (\textsf{Group}) The $n$-dimensional \emph{signflip group}, denoted as
$\pm(n)$, is defined by
\begin{equation}
\pm(n)=\{-1,1\}^{n}.
\end{equation}

\item (\textsf{Action}) The \emph{action} of $\sigma\in\pm(n)$ on $p\in$
$\mathbb{R}[x_{1},\ldots,x_{n}]$, denoted as $\sigma\bullet p$, is defined by%
\[
(\sigma\bullet p)(x)=p(\operatorname{diag}(\sigma)x).
\]

\item (\textsf{Equivalence}) We say that $f,g\in\mathbb{R}[x_{1},\ldots
,x_{n}]$ are \emph{signflip equivalent}, and write $f\sim_{\pm(n)}g$, if
\[
\exists_{\sigma\in\pm(n)}\ g=\sigma\bullet f.
\]

\end{enumerate}
\end{definition}

\noindent Now we are ready to state the \emph{key} theoretical result of this section.
\begin{theorem}[Key result: orthogonal equivalence via PW-PCA]\label{thm:ortho:equiv} Suppose $f, g
\in\mathbb{R}[x_{1},\dots,x_{n}]$ both satisfy the generic
\cref{cond:simple:eigvals}, and let $(\lambda_{f}, V_{f})$ and $(\lambda_{g},
V_{g})$ be PW-PCA's for $f$ and $g$, respectively. Then we have

\begin{enumerate}
\item[\textrm{\bf C1}:] $f \sim_{O(n)} g\;\;\;\implies\;\;\;\lambda_{f}
= \lambda_{g}$,

\item[\textrm{\bf C2}:] $f \sim_{O(n)} g \;\;\;\iff\;\;\;\hat{f}
\sim_{\pm(n)} \hat{g}$,

where $\hat{f} = V_{f} \bullet f$ and $\hat{g} = V_{g} \bullet g$.
\end{enumerate}

\noindent Moreover we also have

\begin{enumerate}
\item[\textrm{\bf C3}:] $\sigma\bullet\hat{f} = \hat{g}\;\;\;\implies
\;\;\; R \bullet f = g$ \;\;\;where $R = V_{f} \operatorname{diag}(\sigma)
V_{g}^{\top}$.
\end{enumerate}
\end{theorem}
Before proceeding to the proof of Theorem \ref{thm:ortho:equiv} we collect several important observations in Remark \ref{remark:orthEq} below.
\begin{remark}\label{remark:orthEq}
\strut
\begin{itemize}
  \item[{\bf C1}:] This claim provides a partial check on  orthogonal equivalence  between $f$ and $g$.  Taking the contrapositive, it says  that if $\lambda_{f} \neq \lambda_{g}$ then
 $f$ and $g$ are \emph{not} orthogonally equivalent. 
 
Since $\lambda$'s lie in a real algebraic extension of the coefficient field of $f$ and $g$, the check involves comparing real algebraic numbers.  As usual one can avoid working with real algebraic numbers by instead comparing the coefficients of the characteristic polynomials of the underlying covariant matrices.
 
  \item[{\bf C2}:] This claim provides a complete check on orthogonal equivalence between $f$ and $g$.
 If one  tries to check $f \sim_{O(n)} g$ using only  the definition, 
  then  it  would require   searching for $R$ 
  in the \emph{infinite}  set $O(n)$ such that $R\bullet f =g$.
  Of course, it is algorithmically impossible.
 This  claim tells us that we can check instead   
$\hat{f} \sim_{\pm(n)} \hat{g}$, which would require 
searching for $\sigma$ in the \emph{finite} set $\pm(n)$ 
such that $ \sigma \bullet \hat{f} = \hat{g}  $.
Of course, it is algorithmically possible.
Thus this claim can be viewed as reducing the infinite search space $O(n)$ to the finite
search space $\pm(n)$. 

This claim can be also viewed as stating the following:   for every $h\in \mathbb{R}[x_1,\ldots,x_n]$, the polynomial   $\hat{h}$ is   the \emph{canonical form} of the orbit of $h$ upto sign and
that   $V_{h}$ is
the \emph{canonical simplifier}, that is, an element of $O(n)$ that acts on $h$ to transform it to its canonical form $\hat{h}$.
  
  \item[{\bf C3}:] This claim shows a simple way to  find  an $R \in O(n)$   such that $R \bullet f = g$ when$ \hat{f} \sim_{\pm(n)} \hat{g}$:  Find $\sigma$ such that $\sigma\bullet\hat{f} = \hat{g}$. Then we have $R=V_{f} \operatorname{diag}(\sigma)
V_{g}^{\top}$.
\end{itemize}
\end{remark}

\bigskip

The remainder of this section will be dedicated to the proof of the key theorem (Theorem~\ref{thm:ortho:equiv}).
The proof of this theorem is split into several smaller lemmas, namely Lemmas \ref{thm:unique:pwaxes}, \ref{lem:hrf=hrhf}, \ref{thm:equiv:pwvarcov}, \ref{lem:cov(Rf)=Rcov(f)}, and \ref{thm:equiv:pwaxes},  which readers may also find meaningful on their own. 
Before proving the lemmas, we will provide a bird's eye view 
of how the lemmas are used, showing their top-down dependencies. 
\begin{itemize}
\item The proof of Theorem~\ref{thm:ortho:equiv} uses 
      Lemma~\ref{thm:unique:pwaxes} and Lemma ~\ref{thm:equiv:pwaxes}.
\item The proof of Lemma~\ref{thm:equiv:pwaxes} 
      uses Lemma~\ref{lem:cov(Rf)=Rcov(f)}.
\item The proof of Lemma~\ref{lem:cov(Rf)=Rcov(f)} uses 
       Lemma~\ref{lem:hrf=hrhf} and Lemma~\ref{thm:equiv:pwvarcov}. 
\end{itemize}

\bigskip

Now let us begin with Lemma~\ref{thm:unique:pwaxes}.
It states that, for almost all $f$, its principal variances  are unique and its principal
axes are unique up to signflipping.
 
\begin{lemma}
[Uniqueness of PW-PCA up to signflips]\label{thm:unique:pwaxes}
Suppose $f\in\mathbb{R}[x_{1},\dots,x_{n}]$ satisfies the generic \cref{cond:simple:eigvals}.
If $(\lambda_{1},V_{1})$ and $(\lambda_{2},V_{2})$
are two PW-PCA's of $f$, then
\[
\lambda_{1}=\lambda_{2}\text{\; \; \; and \; \; \; }\exists_{\sigma\in\pm(n)}%
\ \ V_{2}=V_{1}\operatorname*{diag}\sigma.
\]
\end{lemma}

\begin{proof}
[Proof of \cref{thm:unique:pwaxes}] From the definition of PW-PCA, it is immediate that $\lambda_1 = \lambda_2$. Since $f$ satisfies the generic assumption, the eigenvalues of $C_f$ are distinct. Thus for each eigenvalue, the corresponding eigenspace has dimension one. Hence the eigenvector is unique up to sign. 
Thus, there exist some sign-flips $\sigma\in\{-1,1\}^{n}$ so that $V_{2}%
=V_{1}\operatorname*{diag}\sigma.$
\end{proof}

\bigskip

\noindent

 The following  four lemmas address a    natural  question: How do the following objects behave under the orthogonal group action?
 
\begin{itemize}
\item[] $\quad \bullet$ Homogenization (Lemma~\ref{lem:hrf=hrhf})
\item[] $\quad \bullet$ PW-covariance of homogeneous polynomials  (Lemma~\ref{thm:equiv:pwvarcov})
\item[] $\quad \bullet$ PW-covariance of homogenized polynomials (Lemma~\ref{lem:cov(Rf)=Rcov(f)})
\item[] $\quad \bullet$ PW-PCA (Lemma~\ref{thm:equiv:pwaxes})
\end{itemize}
The lemmas  answer that  each  of them behaves \emph{equivariantly} under the orthogonal group action.
\begin{lemma}[Orthogonal equivariance of homogenization]
\label{lem:hrf=hrhf}Let $f\in\mathbb{R}%
[x_{1},\dots,x_{n}]$ and $R\in O\left(  n\right)  $.
Then we have
\[
\overline{R\bullet f}\;\;=\;\;\overline{R}\bullet\overline{f}
\quad \text{where} \quad
\overline{R}=%
\begin{bmatrix}
R & \\
& 1
\end{bmatrix}
\in O(n+1).
\]

\end{lemma}

\begin{proof}
Let $d=\deg f$. Then $d=\deg g$. Note
\begin{align*}
\left(  \overline{R\bullet f}\right)  (\overline{x})  &  =\left(  \overline
{R\bullet f}\right)  (x,x_{n+1})&&\text{where\ }\overline{x}%
=(x,x_{n+1})\text{ }\\
&  =\left(  R\bullet f\right)  (x/x_{n+1})x_{n+1}^{d}&&\text{from
the definition of homogenization (\ref{eq:homo})}\\
&  =f(R(x/x_{n+1}))x_{n+1}^{d}&&\text{since }g=R\bullet f\\
&  =f((Rx)/x_{n+1})x_{n+1}^{d}&&\text{since }R\left(  x/x_{n+1}%
\right)  =\left(  Rx\right)  /x_{n}\\
&  =\overline{f}(Rx,x_{n+1})&&\text{from the definition of homogenization
(\ref{eq:homo})}\\
&  =\overline{f}(\overline{R}\overline{x})&&\text{from the definition of }\overline{R}\\
&  =\overline{R}\bullet\overline{f}\left(  \overline{x}\right)  &&\text{since }%
\overline{R}\in O\left(  n+1\right).
\end{align*}
Note that the above equality holds for every $\overline{x}$. Thus we have
\[
\overline{R\bullet f}\;\;=\;\;\overline{R}\bullet\overline{f}.
\]
We have proved the lemma.
\end{proof}

We now define the action of an element $R$ of $O(n)$ on a symmetric matrix.

\begin{definition}[Orthogonal action on a symmetric matrix]\label{def:RonM}
The \emph{action} of $R\in O\left(  n\right)  $ on
 a symmetric $M\in \mathbb{R}^{n\times n}$, denoted as $R\bullet M$, is defined by
\[
R\bullet M \;=\; R^{\top} M R.
\]
This is the standard definition when  $M$ is viewed as a representation of  a quadratic  form.
\end{definition}
We will employ the definition above in the statement and proof of the Lemma below. Recall that for a homogeneous polynomial $f$ the polynomial weighted covariance, ${\rm Cov}(f)$, is a real symmetric matrix. 

\begin{lemma}
[Orthogonal equivariance of PW-covariance of homogeneous polynomials] Let $f\in
\mathbb{R}[x_{1},\dots,x_{n}]$ be \label{thm:equiv:pwvarcov}a homogeneous
polynomial and let $R\in O\left(  n\right)  $. Then we have%
\[
\operatorname{Cov}(R\bullet f)\ \ =\ \ R \bullet \operatorname{Cov}(f).
\]
\end{lemma}

\begin{proof}
[Proof of \cref{thm:equiv:pwvarcov}]Note%
\begin{align*}
\operatorname{Cov}(R\bullet f)  &  =\int_{\mathbb{S}^{n-1}}f(Rx)^{2}%
\,xx^{\top}\;d\mu(x)&&\text{from Definition\ \ref{def:pwcov}}\\
&  =\int_{\mathbb{S}^{n-1}}f(Rx)^{2}\,(R^{\top}R)xx^{\top}(R^{\top}%
R)\;d\mu(x)&&\text{since }R^{\top}R=I\\
&  =\int_{\mathbb{S}^{n-1}}f(Rx)^{2}\,R^{\top}(Rx)(Rx)^{\top}R\;d\mu
(x)&&\text{by re-associating}\\
&  =R^{\top}\left(  \int_{\mathbb{S}^{n-1}}f(Rx)^{2}\,(Rx)(Rx)^{\top}%
\;d\mu(x)\right)  R&&\text{since }R\text{ does not depend on }x\\
&  =R^{\top}\left(  \int_{R\mathbb{S}^{n-1}}f(y)^{2}\,yy^{\top}\;d\mu(R^{\top
}y)\right)  R&&\text{by the change of variable }y=Rx\\
&  =R^{\top}\left(  \int_{\mathbb{S}^{n-1}}f(y)^{2}\,yy^{\top}\;d\mu
(y)\right)  R&&\text{since }d\mu(R^{\top}y)=d\mu(y)\\
&  =R^{\top}\operatorname{Cov}(f)R.&&\text{from
Definition\ \ref{def:pwcov}} \\
&  = R \bullet \operatorname{Cov}(f)  &&\text{from
Definition\ \ref{def:RonM}}.
\end{align*}
We have proved the lemma.

\end{proof}
 The next lemma is analogous to Lemma \ref{thm:equiv:pwvarcov} above except for the case where we start with a non-homogenous polynomial and homogenize. 
\begin{lemma}[Orthogonal equivariance of  PW-covariance of homogenized polynomials]
\label{lem:cov(Rf)=Rcov(f)}Let $f\in
\mathbb{R}[x_{1},\dots,x_{n}]$ and $R\in~O\left(  n\right)  $. Then we have%
\[
\operatorname{Cov}(\overline{R\bullet f})_{1:n,1:n}\;\;\;\;=\;\;\;\;R \bullet \operatorname{Cov}
(\overline{f})_{1:n,1:n}
\]

\end{lemma}

\begin{proof}
Note%
\begin{align*}
\operatorname{Cov}(\overline{R\bullet f})_{{}} &  =\operatorname{Cov}(\overline{R}\bullet\overline{f})&&\text{from Lemma \ref{lem:hrf=hrhf}}\\
&  =\overline{R}^{\top}\operatorname{Cov}(\overline{f})\overline{R}&&\text{from
Lemma \ref{thm:equiv:pwvarcov}}\\
&  =\left[
\begin{array}
[c]{ll}%
R & \\
& 1
\end{array}
\right]  ^{\top}\left[
\begin{array}
[c]{ll}%
\operatorname{Cov}(\overline{f})_{1:n,1:n} & \operatorname{Cov}(\overline{f}%
)_{1:n,n+1}\\
\operatorname{Cov}(\overline{f})_{n+1,1:n} & \operatorname{Cov}(\overline{f})_{n+1,n+1}%
\end{array}
\right]  \left[
\begin{array}
[c]{ll}%
R & \\
& 1
\end{array}
\right]  &&\text{by showing the parts}\\
&  =\left[
\begin{array}
[c]{rr}%
R^{\top}\operatorname{Cov}(\overline{f})_{1:n,1:n}\;\;R & R^{\top}\operatorname{Cov}(\overline{f})_{1:n,n+1}\;\,\\ 
\operatorname{Cov}(\overline{f})_{n+1,1:n}R & \operatorname{Cov}(\overline{f})_{n+1,n+1}%
\end{array}
\right]  &&\text{by multiplying.}%
\end{align*}
By equating the top-right parts, we have
\[
\operatorname{Cov}(\overline{R\bullet f})_{1:n,1:n}\;\;\;=\;\;\;R^{\top}\operatorname{Cov}%
(\overline{f})_{1:n,1:n}\;R \;\;\;=\;\;\; R \bullet \operatorname{Cov}(\overline{f})_{1:n,1:n}.
\]
We have proved the lemma.
\end{proof}

We now prove the final lemma which will establish the orthogonal equivariance of PW-PCA  needed to obtain a proof for Theorem \ref{thm:ortho:equiv}.
\begin{lemma}
[Orthogonal equivariance of PW-PCA]\label{thm:equiv:pwaxes}Let $f\in
\mathbb{R}[x_{1},\dots,x_{n}]$ and $R\in O\left(n\right)$. Then the followings are equivalent.

\begin{enumerate}
\item $(\lambda,V)$ is a PW-PCA of $f$.

\item $(\lambda,R^{\top}V)$ is a PW-PCA of $R\bullet f$.
\end{enumerate}
\end{lemma}

\begin{proof}
Suppose that $(\lambda,V)\ \text{is\ a\ PW-PCA\ of\ }f$. Then the following are equivilent:
\begin{align*}
&  (\lambda,V)\ \text{is\ a\ PW-PCA\ of\ }f\ \ \ \\
\Longleftrightarrow\ \ \  &  V\operatorname{diag}(\lambda)V^{\top
}=\operatorname{Cov}(\overline{f})_{1:n,1:n}&&\text{from Theorem
\ref{thm:pwpca:cov}}\\
\Longleftrightarrow\ \ \  &  R^{\top}V\operatorname{diag}(\lambda)V^{\top
}R=R^{\top}\operatorname{Cov}(\overline{f})_{1:n,1:n}R&&\text{since
}R\ \text{is invertible}\\
\Longleftrightarrow\ \ \  &  \left(  R^{\top}V\right)  \operatorname{diag}%
(\lambda)\left(  R^{\top}V\right)  ^{\top}=R^{\top}\operatorname{Cov}%
(\overline{f})_{1:n,1:n}R&&\text{since }V^{\top}R=\left(  R^{\top
}V\right)  ^{\top}\\
\Longleftrightarrow\ \ \  &  \left(  R^{\top}V\right)  \operatorname{diag}%
(\lambda)\left(  R^{\top}V\right)  ^{\top}=\operatorname{Cov}(\overline
{R\bullet f})_{1:n,1:n}&&\text{from Lemma \ref{lem:cov(Rf)=Rcov(f)}%
}\\
\Longleftrightarrow\ \ \  &  (\lambda,R^{\top}V)\ \text{is a PW-PCA of
}R\bullet f&&\text{from Theorem \ref{thm:pwpca:cov}.}%
\end{align*}
Thus we have proved the lemma.
\end{proof}

\bigskip

Now we are ready to prove Theorem~\ref{thm:ortho:equiv}, the key theorem of this section. 
The proof will crucially use  
the uniqueness of PW-PCA up to signflips (Lemma~\ref{thm:unique:pwaxes}) 
and orthogonal equivariance of PW-PCA (Lemma~\ref{thm:equiv:pwaxes}).

\begin{proof}
[Proof of \cref{thm:ortho:equiv}] Let $f, g \in\mathbb{R}[x_{1},\dots,x_{n}]$
have respective PW-PCA's $(\lambda_{f}, V_{f})$ and $(\lambda_{g}, V_{g})$,
where both~$f$ and~$g$ satisfy \cref{cond:simple:eigvals}. We prove each claim
one by one:

\begin{enumerate}
\item[\textrm{{\textbf{ C1.}}}] $f \sim_{O(n)} g \;\;\;\implies\;\;\;\lambda_{f} =
\lambda_{g}$.

Suppose that $f \sim_{O(n)} g$. Then there exists $R \in O(n)$ such that $R
\bullet f =g$.
It follows from \cref{thm:equiv:pwaxes} that $\lambda_{f}$
are then principal variances for $g$, and it finally follows from
\cref{thm:unique:pwaxes} that $\lambda_{f} = \lambda_{g}$, which completes the
proof of this claim.

\item[\textrm{{\textbf{C2.}}}] $f \sim_{O(n)} g \;\;\;\iff\;\;\;\hat{f} \sim_{\pm(n)}
\hat{g}$.

For the $\implies$ direction, suppose that $f \sim_{O(n)} g$. Then there
exists $R \in O(n)$ such that $R \bullet f = g$. It follows from
\cref{thm:equiv:pwaxes} that $R^{\top}V_{f}$ is then a set of principal axes
for $g$, and it next follows from \cref{thm:unique:pwaxes} that there exists
$\sigma\in\pm(n)$ such that $V_{g} = R^{\top}V_{f} \operatorname{diag}%
(\sigma)$. Thus, $R V_{g} = V_{f} \operatorname{diag}(\sigma)$ and so
\begin{align*}
(\sigma\bullet\hat{f})(x)  &  = \hat{f}(\operatorname{diag}(\sigma)
x) &&\text{by the definition of $\sigma\bullet\hat{f}$}\\
&  = f(V_{f} \operatorname{diag}(\sigma) x)&&\text{by the definition
of $\hat{f}$}\\
&  = f(R V_{g} x)&&\text{since $R V_{g} = V_{f} \operatorname{diag}%
(\sigma)$}\\
&  = g(V_{g} x)&&\text{since $R \bullet f = g$}\\
&  = \hat{g}(x)&&\text{by the definition of $\hat{g}$} .
\end{align*}
Hence, we have that $\hat{f} \sim_{\pm(n)} \hat{g}$, which completes the proof
of this direction.

For the $\impliedby$ direction, suppose that $\hat{f} \sim_{\pm(n)} \hat{g}$.
Then there exists $\sigma\in\pm(n)$ such that $\sigma\bullet\hat{f} = \hat{g}$
and so
\begin{align*}
g(x)  &  = g(V_{g} V_{g}^{\top}x)&&\text{since $V_{g} V_{g}^{\top}=
I_{n}$}\\
&  = \hat{g}(V_{g}^{\top}x)&&\text{by the definition of $\hat{g}$}\\
&  = \hat{f}(\operatorname{diag}(\sigma) V_{g}^{\top}x)&&\text{since
$\sigma\bullet\hat{f} = \hat{g}$}\\
&  = f(V_{f} \operatorname{diag}(\sigma) V_{g}^{\top}x)&&\text{by
the definition of $\hat{f}$}\\
&  = f(Rx)&&\text{by defining $R = V_{f} \operatorname{diag}(\sigma)
V_{g}^{\top}$}\\
&  = (R \bullet f)(x)&&\text{by the definition of $R \bullet f$} ,
\end{align*}
where $R \in O(n)$ since it is a product of orthogonal matrices. Hence, we
have that $f \sim_{O(n)} g$, which completes the proof of this direction.

\item[\textrm{{\textbf{C3.}}}] $\sigma\bullet\hat{f} = \hat{g} \;\;\;\implies\;\;\; R
\bullet f = g$ \;\;\;\;where $R = V_{f} \operatorname{diag}(\sigma) V_{g}^{\top}$.

Suppose that $\sigma\bullet\hat{f}=\hat{g}$. Then by similar arguments as in
the $\impliedby$ direction of {\bf C2}, we have
\[
(R\bullet f)(x)=f(Rx)=f(V_{f}\operatorname{diag}(\sigma)V_{g}^{\top}x)=\hat
{f}(\operatorname{diag}(\sigma)V_{g}^{\top}x)=\hat{g}(V_{g}^{\top}%
x)=g(V_{g}V_{g}^{\top}x)=g(x),
\]
and so $R\bullet f=g$, which completes the proof of this claim.
\end{enumerate}

\noindent We have proved the key theorem.
\end{proof}

\bigskip

The theory developed so far, in particular Theorem~\ref{thm:ortho:equiv}, immediately leads to 
the following algorithm for tacking the main problem of this paper (Problem~\ref{prb:problem}).

\begin{algorithm}
[\textsf{Certificate for Orthogonal Equivalence}]\label{alg:certificate}\mbox{}\medskip

{\bf Input:} $f,g\in\mathbb{R}[x_{1},\ldots,x_{n}]$ of degree $d$ such that
$n,d\geq2$ and $f\sim_{O(n)}g$\medskip

{\bf Output:} $R\in O(n)$ such that $R\bullet f=g$

\begin{enumerate}
\item $(\lambda_f,V_{f})\leftarrow$\textsf{PW\_PCA}$(f)$

      $(\lambda_g,V_{g})\leftarrow$\textsf{PW\_PCA}$(g)$

\item $\hat{f}\leftarrow V_f \bullet f$

$\hat{g}\leftarrow V_g \bullet g$

\item Find
$\sigma \in \{-1,1\}^n$ such that $\sigma \bullet \hat{f} = \hat
{g}$

\item $R\leftarrow V_{f}\operatorname*{diag}\sigma V_{g}^{\top}$
\end{enumerate}
\end{algorithm}

\medskip

We illustrate the above algorithm on the running  example. 
We executed the algorithm using double precision ($\sim$ 16 significant decimal digits) floating point arithmetic displaying only~3 decimal digits.

\begin{example}
[Running example continued]Again consider the two polynomials $f$ and $g$ of degree $d=3$ in $\RR[x_1, x_2, x_3]$ given below:\small
\begin{align*}
f&=-27 x_{1}^{3}+27 x_{2}^{2} x_{3}-9 x_{3},\\
g&=-12 x_{1}^{3}+12 x_{1}^{2} x_{2}-12 x_{1}^{2} x_{3}+6 x_{1} x_{2}^{2}+36
x_{1} x_{2} x_{3}-33 x_{1} x_{3}^{2}+9 x_{2}^{3}-6 x_{2}^{2} x_{3}+6 x_{2}
x_{3}^{2}-6 x_{3}^{3}+3 x_{1}-6 x_{2}-6 x_{3} .
\end{align*}\normalsize
We now apply Algorithm \ref{alg:certificate} with input $f$ and $g$ as above. From this procedure we obtain the element $R\in O(n)$ certifying the orthogonal equivalence of $f$ and $g$.

\begin{enumerate}

\item $\lambda_{f},V_{f}\leftarrow[ 808.346, \; 212.351, \; 170.966], \left[
\begin{array}
[c]{rrr}%
0.999 & 0.000 & 0.047\\
0.000 & - 1.000 & 0.000\\
- 0.047 & 0.000 & 0.999
\end{array}
\right]  $

$\lambda_{g},V_{g}\leftarrow[ 808.345,\; 212.351, \;170.966], \left[
\begin{array}
[c]{rrr}%
0.682 & 0.667 & - 0.302\\
- 0.364 & 0.667 & 0.650\\
0.635 & - 0.333 & 0.697
\end{array}
\right]  $

\item $\hat{f}\leftarrow- 26.910 x_{1}^{3} - 3.806 x_{1}^{2} x_{3} - 1.271
x_{1} x_{2}^{2} - 0.179 x_{1} x_{3}^{2} + 26.970 x_{2}^{2} x_{3} - 0.003
x_{3}^{3} + 0.424 x_{1} - 8.990 x_{3}  +{\rm terms \; with \; coefficients \; less\; than \;}10^{-3}$

$\hat{g}\leftarrow- 26.910 x_{1}^{3} - 3.806 x_{1}^{2} x_{3} - 1.271 x_{1}
x_{2}^{2} - 0.179 x_{1} x_{3}^{2} + 26.970 x_{2}^{2} x_{3} - 0.003 x_{3}^{3} +
0.424 x_{1} - 8.990 x_{3} +{\rm terms \; with \; coefficients \; less\; than \;}10^{-3}$

\item $\sigma\leftarrow[1, -1, 1]$

\item $R\leftarrow\left[
\begin{array}
[c]{rrr}%
0.667 & -0.333 & 0.667\\
0.667 & 0.667 & -0.333\\
-0.333 & 0.667 & 0.667
\end{array}
\right]  $
\end{enumerate}
We may easily check that $R$ is an element of $O(n)$ and that $f(Rx)=g(x)$, up to floating point arithmetic error, and hence $R$ is indeed the desired certificate. 
\end{example}

We conclude the section with a collection of remarks on the methods presented above. 
\begin{remark} \mbox{} Below are several observations regarding Algorithm \ref{alg:certificate}.
\label{remark:OrthEqImplment}
\begin{enumerate}[a)]
  \item Note that $\lambda_f$ and $\lambda_g$ from the {\rm PW-PCA} are not used in Algorithm \ref{alg:certificate} for finding a             certificate.
We chose to include them since they might be useful for tackling some other problems.
Observe that, since $f$ and $g$ are equivalent, Theorem~\ref{thm:ortho:equiv}, {\rm \bf C1}, ensures that $\lambda_f=\lambda_g$,
as can be seen from the numerical values for them in Step 1 of the above example.

\medskip

  \item Even if it is unnecessary, let us carry out a sanity check on the output of the above example since we executed the algorithm  using floating point arithmetic. Easy computation shows 
  \[\left\vert \left\vert R\bullet f-g\right\vert\right\vert = 2.035 \times 10^{-13}.\]
Recalling that we used floating points with ~16 significant decimal digits, we see that the error is negligibly small. In fact if one uses higher precision floating points arithmetic ($\sim$ 30 significant decimal digits), then the error becomes even smaller: $3.090 \times 10^{-27}$.

  \item  If the sanity check fails (that is, $||R \bullet f - g||$ is not negligible) then, according to the key theorem (Theorem~\ref{thm:ortho:equiv}), it is due to the one (or more) of the following:
  At least one of~$f$ and~$g$ does not satisfy the generic assumption (Assumption~\ref{cond:simple:eigvals}) or the user, by mistake, inputted~$f$ and~$g$ which are not orthogonally equivalent. Since it is highly unlikely that $f$ and $g$ do not satisfy the genericity condition (which can be checked by examining $\lambda_f$ and $\lambda_g$), one can conclude that it is highly likely that the user mistakenly inputted $f$ and $g$ which are not orthogonally equivalent. 
  
  \item Thus the above algorithm could be also used as an algorithm for checking the orthogonal equivalence of two arbitrary polynomials $f$ and $g$ under the generic condition, if all the computations were carried out exactly. In practice to carry out such exact computations we would need to work in algebraic number fields which can become computationally expensive in practice. 
  
\end{enumerate}
\end{remark}

\section{Performance of the proposed algorithm}
\label{sec:performance}

This section investigates the performance of the proposed algorithm through computer experiments.
Before we discuss the results, we first describe the overall setup that was used.

\paragraph{\underline{Setup for the experiments}}
\begin{itemize}
\item {\sf Program}: 
We implemented the proposed algorithm in the Julia programming language \cite{bezanson2017julia}
using the standard library and the following packages:
Nemo \cite{nemojl},
Combinatorics,
and
GenericLinearAlgebra.
Since our primary goal was to investigate the algorithmic aspects of the performance,
we used a mostly straightforward preliminary implementation of the algorithm.
For the baseline algorithm (described below), we used the AlgebraicSolving package,
which employs the highly-optimized msolve library \cite{berthomieu}.

We also used the following Julia packages to run the experiments and export the results as tables and figures:
CacheVariables,
CairoMakie \cite{DanischKrumbiegel2021},
Dictionaries,
LaTeXStrings,
LaTeXTabulars,
Pluto,
ProgressLogging,
StableRNGs, and
TerminalLoggers.

\item {\sf Benchmark}: For each pair $(n,d)$, we generated ten $(f,g)$ pairs as follows:
    \begin{enumerate}
      \item Generate $f \in \mathbb{Z}[x_1,\ldots,x_n]$ homogeneous of degree $d$ whose coefficients are randomly taken from the range $-100$ to $100$, omitting zero. This produces a dense polynomial.
      \item Generate an orthogonal $R \in \mathbb{Q}^{n \times n}$ using the Cayley transform~\cite{Cayley1846} as follows:
      Generate a random $n \times n$ antisymmetric matrix $S \in \mathbb{Z}^{n \times n}$ with entries between $-10$ and $10$ and a random sign vector $\sigma \in \{-1,1\}^n$.
      Then compute $R = (S-I_n)^{-1} (S+I_n) \operatorname{diag}(\sigma)$.
      \item Compute $g = R \bullet f \in \mathbb{Q}[x_1,\ldots,x_n]$.
   \end{enumerate}
   We generated $f$ and $g$ with exact rational coefficients because the baseline algorithm (described below) involves solving a highly over-determined system of polynomial equations, where even small numeric errors in the coefficients would likely almost always make the system inconsistent. Notably, the proposed algorithm (\cref{alg:certificate}) does \emph{not} involve solving an over-determined system and does not have this limitation. It can be safely run with floating point numbers.

\item {\sf Timing}: All timings were measured as wall-clock time, and we report median runtimes across the ten $(f,g)$ pairs.

\item {\sf Computer}: The experiments were run on a 2021 Macbook Pro with an M1 Max CPU (64GB of RAM).

\item {\sf Reproducible research}: Code to reproduce the experiments is available online at \url{https://gitlab.com/dahong/ortho-equiv-pwpca}.
\end{itemize}

\paragraph{\underline{Comparison with the baseline method}}\strut

\medskip

\noindent To the best of our knowledge, there is currently no publicly available software for finding certificates of orthogonal equivalence.
Thus, we implemented the following straightforward algorithm as a baseline method for comparison.

\medskip
\begin{algorithm}[Baseline]\label{alg:baseline}\mbox{}
\medskip

{\bf Input:} $f,g\in\mathbb{R}[x_{1},\ldots,x_{n}]$ of degree $d$ such that
$n,d\geq2$ and $f\sim_{O(n)}g$\medskip

{\bf Output:} $R\in O(n)$ such that $R\bullet f=g$

\begin{enumerate}
\item $M\leftarrow\left[
\begin{array}
[c]{lll}%
r_{11} & \cdots & r_{1n}\\
\vdots &  & \vdots\\
r_{n1} &  & r_{nn}%
\end{array}
\right]  \ \ $where $r_{ij}$ are indeterminates

\item $h\leftarrow M \bullet  f  -g \;\; \in\mathbb{R}%
\left[  r_{11},\ldots,r_{nn}\right]  \left[  x_{1},\ldots,x_{n}\right]  $

\item $H\leftarrow$the set of all the coefficients of $h$ and all the entries
of $M^{\top}M-I_n \subset \mathbb{R}%
\left[  r_{11},\ldots,r_{nn}\right]$ 
\item $S\leftarrow$ a real solution $\in \mathbb{R}^{n^2}$ of $H=0$

\item $R\leftarrow$ the matrix $\in \mathbb{R}^{n\times n}$ obtained by instantiating  $M$
      on $S$

\end{enumerate}
\end{algorithm}

\noindent The following table shows the median runtime for the proposed method and the baseline method.
Our runs of the baseline method for $n=5, d=10$ did not complete; they crashed on the first trial after running for around 154 minutes (9240 seconds).
We indicate this on the table as ``$>$9240''.

\begin{center}
    \input{tbl_comp.tex}
\end{center}

\noindent We make the following observations:
\begin{itemize}
\item The proposed method is much faster than the baseline method. In fact, for all three values of $n$, the proposed method on $d=10$ was faster than the baseline method on $d=7$, despite having to handle polynomials with many more terms.
\item The proposed method scales much more gradually than the baseline method with respect to both the number of variables $n$ and the degree $d$. As a result, the speedup generally grows with both $n$ and $d$.
\end{itemize}

\paragraph{\underline{Details on the performance of the  proposed algorithm}}\strut

\medskip

\noindent 
We conclude the paper with a detailed investigation of the performance of the proposed algorithm.
In particular, we assess its scalability, the amount of time spent on each step, and identify potential directions for future improvements.
For this purpose, recall that the proposed algorithm (\cref{alg:certificate}) consists of the following steps:
\begin{enumerate}
\item Compute the PW-PCA's of $f$ and $g$.
\item Compute the canonical forms $\hat{f}$ and $\hat{g}$.
\item Find a valid signflip $\sigma$.
\item Compute the certificate $R$.
\end{enumerate}

\medskip
\noindent The following two sets of graphs show the median amount of time spent by each step
on a set of axes with matching axis limits to make it easy to see how the method scales with respect to the number of variables $n$ and the degree $d$. Recall that these are dense polynomials.

\includegraphics[width=0.9\linewidth]{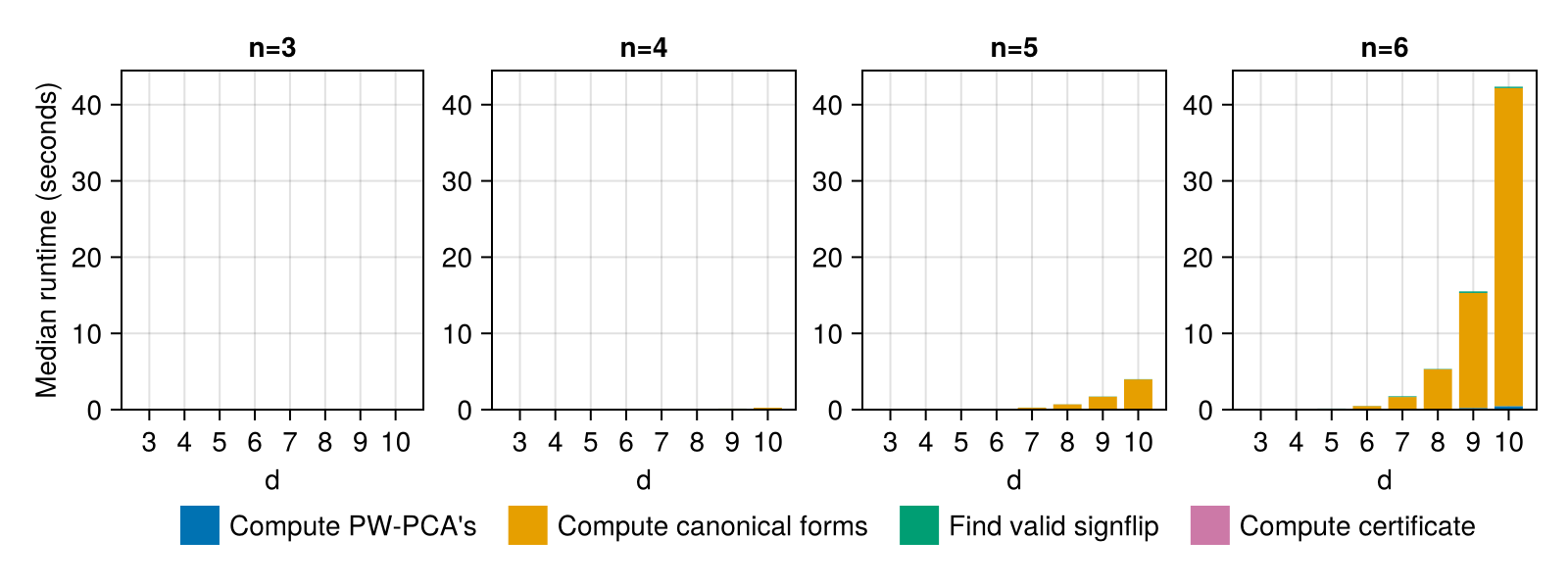}

\includegraphics[width=0.9\linewidth]{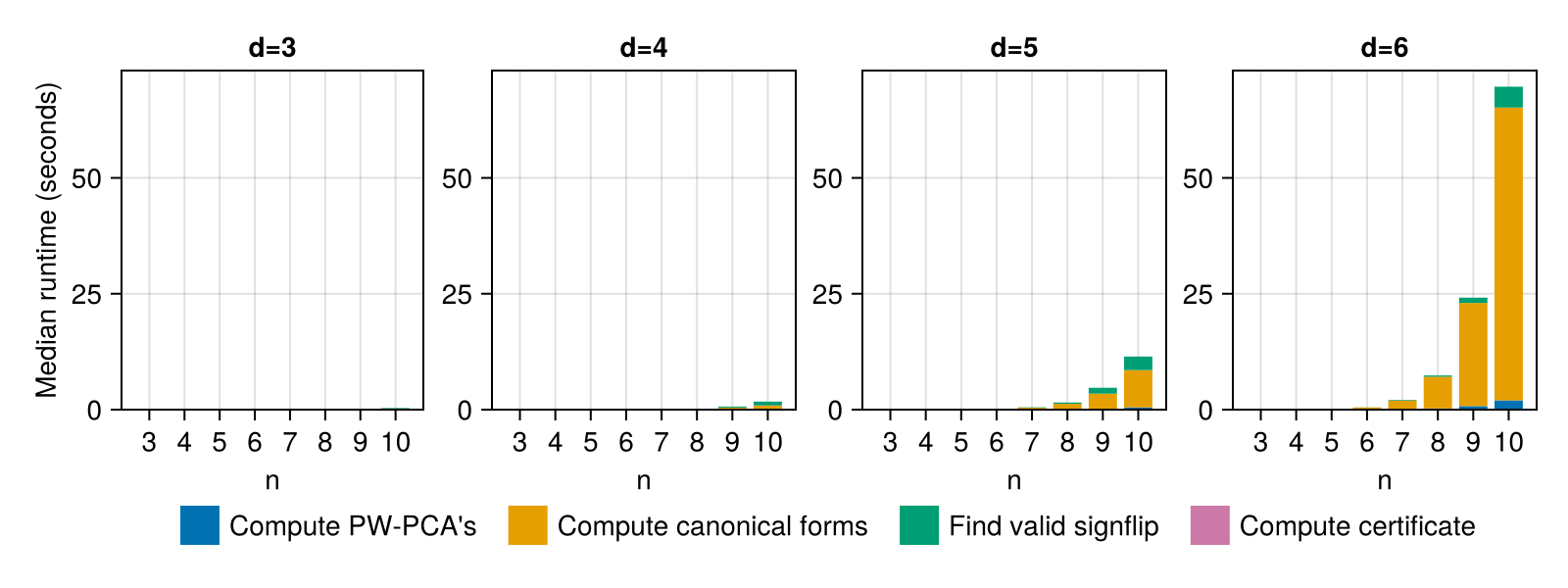}

\noindent We make the following observations:
\begin{itemize}
\item For small $n$ or small $d$, the graphs look essentially empty. This is because the time spent in these cases is so small relative to the larger cases that the bars are not visible.
\item The top row of graphs show that the proposed method scales well to moderate values of $d$ for a fixed small $n$. Likewise, the bottom row of graphs show that the proposed method scales well to moderate values of $n$ for a fixed small $d$. The growing runtimes in both cases appear to be largely driven by the growing cost of computing the canonical forms.
\item For large $n$ and $d$, the graphs in fact suggest that the runtime is essentially dominated by computing the canonical forms.
\end{itemize}
Overall, it seems that the primary barrier to further scalability of the approach is to improve the speed of computing the canonical forms.

\medskip

\noindent To more carefully examine the detailed breakdown of how much time is spent on each step, we plot the same timings again but with the axis limits now tuned for each plot individually.

\includegraphics[width=0.9\linewidth]{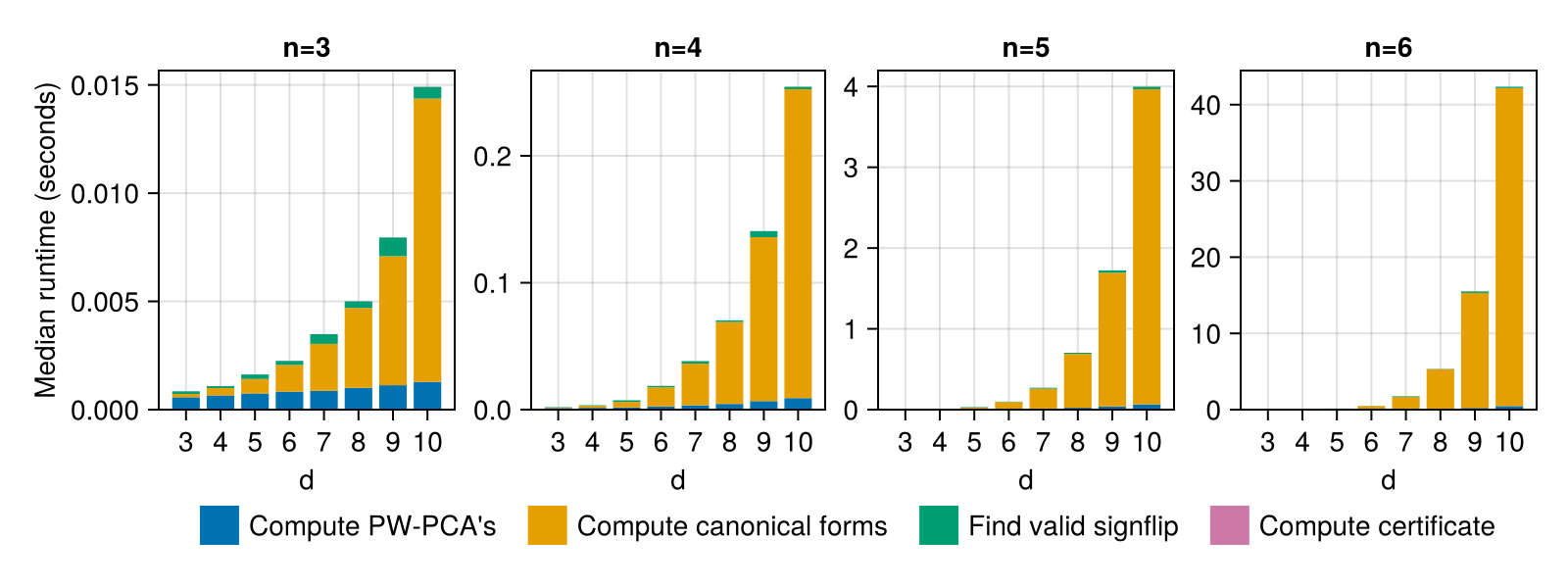}

\includegraphics[width=0.9\linewidth]{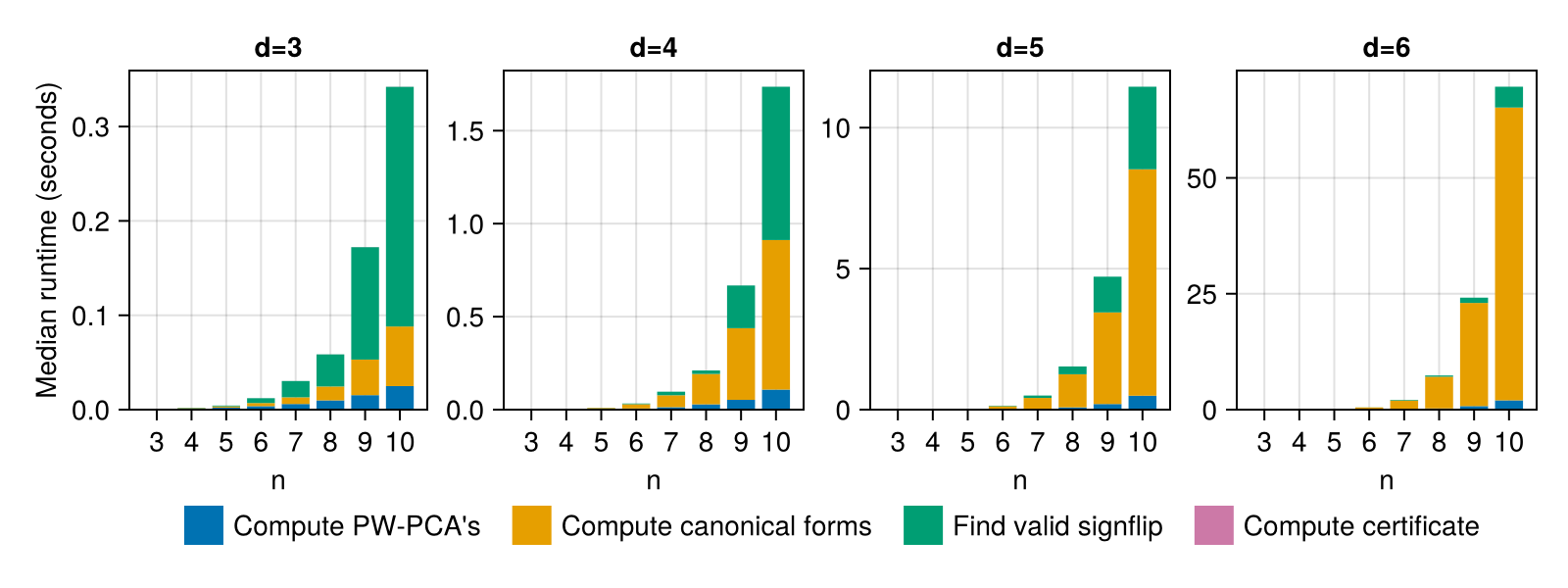}

\noindent We make the following observations:
\begin{itemize}
\item The runtime for PW-PCA is surprisingly negligible, even though the definition of PW-PCA involves global optimization and multivariate integration. This may be due to the effectiveness of \cref{thm:pwcov:formulas}.

\item From the bottom row of graphs, we see that for fixed $d$, finding a valid signflip becomes increasingly dominant as $n$ grows. This may be because the current implementation searches for a valid signflip by using a naive loop over the exponentially large space of $2^n$ possible signflips $\sigma$. It only exits this loop once it finds a signflip $\sigma$ that produces an error between $\sigma \bullet \hat{f}$ and $\hat{g}$ that is within the tolerance. Exploiting structure in $\hat{f}$ and $\hat{g}$ to accelerate this step is a promising direction for future work.

\item From both rows of graphs, we see that for fixed $n$, computing the canonical forms becomes increasingly dominant as $d$ grows. This may be because the current implementation naively expands $f(V_f x)$ and $g(V_g x)$, which can become expensive for large degrees.

Notably, each of these expansions corresponds to a ``symmetric tensor times matrix'' computation involving the corresponding symmetric coefficient tensor and transformation matrix. These types of operations have been heavily studied in the tensor community, e.g.,
in \cite{schatz2014est,daas2025mcp}. It may be possible to accelerate the operation here using similar techniques.

Note also that the canonical forms are only used to check, for given $\sigma$, if $\sigma \bullet V_f \bullet f = V_g \bullet g$, i.e., if $h(x) = \sigma \bullet V_f \bullet f - V_g \bullet g$ is identically zero with respect to $x$. Thus, another promising direction is to avoid the cost of expanding these polynomials alltogether either by random point testing using the Schwartz-Zippel Lemma or by ``generic'' points testing, that is, evaluating $h(x)$ on carefully chosen $x$ values based on some structural properties of $h$.
\end{itemize}
Overall, these plots reveal that obtaining further scalability will indeed likely require first improving the speed of computing the canonical forms then improving the speed of finding a valid signflip.
These are both promising directions for future work.

\paragraph{Acknowledgments:} 
We thank Alexander Demin for helpful suggestions on the use of Julia packages for implementing the proposed algorithm.
We also thank Irina Kogan for helpful discussion on the current state of the art in computational invariant theory.

Martin Helmer is  supported by the Royal Society under grant RSWF\textbackslash R2\textbackslash 242006 and by the United States Air Force Office of Scientific Research (AFOSR) under award
number FA9550-22-1-0462, managed by Dr.~Frederick Leve, and would like to gratefully acknowledge this support.
David Hong's work was partially supported by the US National Science Foundation
Mathematical Sciences Postdoctoral Research Fellowship DMS 2103353.
Hoon Hong's work was partially supported by the US National Science Foundation
CCF 2212461 and CCF 2331401.

\bibliographystyle{abbrv}
\bibliography{refs}
\begin{appendices}
\section{Proof of \cref{prp:genericity}}
\label{appendix:A}

We will prove the genericity of
Assumption~\ref{cond:simple:eigvals}. Let $n\geq2$ and $d\geq2$ be arbitrary
but fixed. The set of all polynomials in $x_{1},\ldots,x_{n}$ of degree at
most~$d$ forms a $\mathbb{R}$-vector space, which is in one-to-one
correspondence with the set of all possible coefficients, that is,
$\mathbb{R}^{m}\ $where $m$ is the number of all monomials in $x_{1}%
,\ldots,x_{n}$ of degree $d$. Let $a$ stand for the vector of coefficients of
$f$ and let
\[
E=\left\{  a\in\mathbb{R}^{m}:C_{f}\ \ \text{has a multiple eigenvalue}%
\right\}  .
\]
It is sufficient to show that $E$ is contained in a Zariski closed set. For
this, it suffices to find $h\in\mathbb{R}\left[  a\right]  \backslash\left\{
0\right\}  $ such that $E\subseteq\left\{  a\in\mathbb{R}^{m}:h\left(
a\right)  =0\right\}  $. We propose%
\[
h=\operatorname*{discriminant}\nolimits_{\lambda}\left(  p\right)
\ \in\mathbb{R}\left[  a\right]  \ \ \ \text{where\ \ \ }p=\left\vert \lambda
I-C_{f}\right\vert \in\mathbb{R}\left[  a\right]  \left[  \lambda\right]  .
\]
From the elementary theory of linear algebra and discriminants, we have
$h=\prod\limits_{1\leq i<j\leq n}\left(  \lambda_{i}-\lambda_{j}\right)  ^{2}$
where $\lambda_{i}$'s are the eigenvalues of $C_{f}$. Thus it is immediate
that $E\subseteq\left\{  a\in\mathbb{R}^{m}:h\left(  a\right)  =0\right\}  $.
Now it only remains to show that $h\neq0$ (not identically vanishing). For
this, it suffices to find $a^{\ast}\in\mathbb{R}^{m}$, equivalently $f^{\ast
},$ such that $C_{f^{\ast}}$ has distinct eigenvalues.

After numerous trial-errors, we propose $a^{\ast}$ corresponding to the
following particular polynomial $f^{\ast}$:
\[
f^{\ast}=\sum_{1\leq s\leq n}g_{s}x_{1}^{d-2}x_{s}^{2}\ \ \ \ \text{where
}g_{s}=n+1-s
\]
Note that every exponent in $\left(  \overline{f^{\ast}}\right)  ^{2}$ is
even. Thus, from Theorem~\ref{thm:pwcov:formulas}, we see immediately that $C_{f^{\ast}}$ is a
diagonal matrix. Hence the eigenvalues of $C_{f^{\ast}}$ are its diagonal
elements. Thus it suffices to show that the diagonal elements of $C_{f^{\ast}%
}$ are distinct. In fact we will show that $C_{f^{\ast},11}>C_{f^{\ast}%
,22}>\cdots>C_{f^{\ast},nn}$.

Applying Theorem~\ref{thm:pwcov:formulas} on $\overline{f^{\ast}}$ and carrying out elementary but
tedious calculations and simplifications, one can obtain the following
expression for the diagonal entries of $C_{f^{\ast}}$:%
\[
\frac{C_{f^{\ast},ii}}{W}=\left\{
\begin{array}
[c]{lll}%
\left(  2d+1\right)  \left(  2d-1\right)  \left(  2d-3\right)  g_{1}%
^{2}\ +\ 2\left(  2d-1\right)  \left(  2d-3\right)  g_{1}u\ +\ 2\left(
2d-3\right)  v+\ \left(  2d-3\right)  u^{2} & \text{if} & i=1\\
\left(  2d-1\right)  \left(  2d-3\right)  g_{1}^{2}\ +\ 4\left(  2d-3\right)
g_{1}g_{i}\ +\ 8g_{i}^{2}\ +\ 2\left(  2d-3\right)  g_{1}u\ +\ 4g_{i}%
u+\ 2v\ +\ u^{2} & \text{if} & i\geq2
\end{array}
\right.
\]
where%
\[
W=\frac{\pi^{\left(  n+1\right)  /2}(2d-5)!!}{2^{d}\Gamma(\frac{n+1}{2}%
+d+1)},\ \ \ \ \text{\ }u=g_{2}+\cdots+g_{n},\ \ \ \ \ \text{and\ \ }%
v=g_{2}^{2}+\cdots+g_{n}^{2}.
\]
One can immediately see that $C_{f^{\ast},22}>C_{f^{\ast},33}>\cdots
>C_{f^{\ast},nn}$ by inspecting the above expression of $C_{f^{\ast},ii}$ for
$i\geq2$ and recalling $g_{2}>g_{3}>\cdots>g_{n}$. Thus it remans to show
$C_{f^{\ast},11}>C_{f^{\ast},22}$. For this, note%
\begin{align*}
&  \ \left(  C_{f^{\ast},11}-C_{f^{\ast},22}\right)  /W\\
= &  \ \underset{}{\underset{>0}{\underbrace{2d\left(  2d-1\right)  \left(
2d-3\right)  }}}g_{1}^{2}\ \;\;\;+\underset{>0}{\underbrace{2\left(
2d-2\right)  \left(  2d-3\right)  }}g_{1}u+\underset{\geq
0}{\underbrace{2\left(  2d-4\right)  }}v\ \;+\left(  2d-4\right)
u^{2}-4\left(  2d-3\right)  g_{1}g_{2}-8g_{2}^{2}-4g_{2}u\\
> &  \ 2d\left(  2d-1\right)  \left(  2d-3\right)
\textcolor{red}{g_{1}g_{2}}+2\left(  2d-2\right)  \left(  2d-3\right)
\textcolor{red}{g_{2}u}+2\left(  2d-4\right)
\textcolor{red}{g_{2}^{2}}+\left(  2d-4\right)  u^{2}-4\left(  2d-3\right)
g_{1}g_{2}-8g_{2}^{2}-4g_{2}u\\
&  \ \text{since }g_{1}^{2}>g_{1}g_{2},\ \ g_{1}u>g_{2}u,\ \ \text{and }v\geq
g_{2}^{2}\\
= &  \ \underset{>0}{\underbrace{\left(  2d\left(  2d-1\right)  \left(
2d-3\right)  -4\left(  2d-3\right)  \right)  }}g_{1}g_{2}+\underset{\geq
0}{\underbrace{\left(  2\left(  2d-2\right)  \left(  2d-3\right)  -4\right)
}}g_{2}u+\underset{}{\underset{\geq0}{\underbrace{\left(  2d-4\right)  }}%
}u^{2}+\left(  2\left(  2d-4\right)  -8\right)  g_{2}^{2}\\
> &  \ \left(  2d\left(  2d-1\right)  \left(  2d-3\right)  -4\left(
2d-3\right)  \right)  \textcolor{red}{g_{2}^{2}}\ \;+\left(  2\left(  2d-2\right)
\left(  2d-3\right)  -4\right)  \textcolor{red}{g_{2}^{2}}\;+\left(
2d-4\right)  \textcolor{red}{g_{2}^{2}}+\left(  2\left(  2d-4\right)
-8\right)  g_{2}^{2}\\
&  \ \text{since }g_{1}g_{2}>g_{2}^{2},\ \ g_{2}u\geq g_{2}^{2},\ \ \text{and
}u^{2}\geq g_{2}^{2}\\
= &  \ \underset{}{8\left(  d+1\right)  d\left(  d-2\right)  }g_{2}%
^{2}\ \ \ \ \ \text{by factoring}\\
\geq &  \ 0
\end{align*}
\noindent Thus $C_{f^{\ast},11}>C_{f^{\ast},22}$. Finally we have proved the genericity
of Assumption~\ref{cond:simple:eigvals}.
\qed

\end{appendices}
\end{document}

%% file: tbl_comp.tex
\begin{tabular}{rrr S[table-format=5.3] S[table-format=1.3] S[table-format=5.2]}
\toprule
$n$ & $d$ & \# terms & {Baseline Method (seconds)} & {Proposed Method (seconds)} & {Approx. Speedup} \\
\midrule
\multirow[t]{4}{*}{$3$} & $7$ & $36$ & 0.024 & 0.004 & 6.0x \\
 & $8$ & $45$ & 0.043 & 0.005 & 8.6x \\
 & $9$ & $55$ & 0.075 & 0.008 & 9.4x \\
 & $10$ & $66$ & 0.139 & 0.015 & 9.3x \\
\midrule
\multirow[t]{4}{*}{$4$} & $7$ & $120$ & 0.944 & 0.039 & 24.2x \\
 & $8$ & $165$ & 3.048 & 0.072 & 42.3x \\
 & $9$ & $220$ & 8.933 & 0.141 & 63.4x \\
 & $10$ & $286$ & 26.166 & 0.256 & 102.2x \\
\midrule
\multirow[t]{4}{*}{$5$} & $7$ & $330$ & 133.690 & 0.277 & 482.6x \\
 & $8$ & $495$ & 261.556 & 0.706 & 370.5x \\
 & $9$ & $715$ & 1334.248 & 1.746 & 764.2x \\
 & $10$ & $1001$ & > 9240 & 4.005 & > 2307.1x \\
\bottomrule
\end{tabular}

%% file: mainOn.bbl
\begin{thebibliography}{10}

\bibitem{ballard2025tdd}
G.~Ballard and T.~G. Kolda.
\newblock {\em Tensor {{Decompositions}} for {{Data Science}}}.
\newblock Cambridge University Press, 1 edition, June 2025.

\bibitem{berthomieu}
J.~Berthomieu, C.~Eder, and M.~{Safey El Din}.
\newblock {msolve: A Library for Solving Polynomial Systems}.
\newblock In {\em {2021 International Symposium on Symbolic and Algebraic Computation}}, 46th International Symposium on Symbolic and Algebraic Computation, pages 51--58, Saint Petersburg, Russia, July 2021. {ACM}.

\bibitem{bezanson2017julia}
J.~Bezanson, A.~Edelman, S.~Karpinski, and V.~B. Shah.
\newblock Julia: A fresh approach to numerical computing.
\newblock {\em SIAM review}, 59(1):65--98, 2017.

\bibitem{breloer2025rational}
H.~Breloer.
\newblock Rational invariants of even degree polynomials under the orthogonal group.
\newblock {\em arXiv preprint arXiv:2501.17504}, 2025.

\bibitem{Cayley1846}
A.~Cayley.
\newblock About the algebraic structure of the orthogonal group and the other classical groups in a field of characteristic zero or a prime characteristic.
\newblock {\em J. Reine Angew.}, 32, 1846.

\bibitem{daas2025mcp}
H.~A. Daas, G.~Ballard, L.~Grigori, S.~Kumar, K.~Rouse, and M.~V{\'e}rit{\'e}.
\newblock Minimizing {{Communication}} for {{Parallel Symmetric Tensor Times Same Vector Computation}}, June 2025.

\bibitem{DanischKrumbiegel2021}
S.~Danisch and J.~Krumbiegel.
\newblock {Makie.jl}: Flexible high-performance data visualization for {Julia}.
\newblock {\em Journal of Open Source Software}, 6(65):3349, 2021.

\bibitem{derksen2015computational}
H.~Derksen and G.~Kemper.
\newblock {\em Computational invariant theory}.
\newblock Springer, 2015.

\bibitem{dolgachev2003lectures}
I.~Dolgachev.
\newblock {\em Lectures on invariant theory}.
\newblock Number 296. Cambridge University Press, 2003.

\bibitem{InvariantRingSource}
L.~Ferraro, F.~Galetto, F.~Gandini, H.~Huang, T.~Hawes, M.~Mastroeni, and X.~Ni.
\newblock {InvariantRing: invariants of group actions. Version~2.0}.
\newblock A \emph{Macaulay2} package available at \url{https://github.com/Macaulay2/M2/tree/stable/M2/Macaulay2/packages}.

\bibitem{InvariantRingArticle}
L.~Ferraro, F.~Galetto, F.~Gandini, H.~Huang, T.~Hawes, M.~Mastroeni, and X.~Ni.
\newblock {The InvariantRing package for \emph{Macaulay2}}.
\newblock {\em Journal of Software for Algebra and Geometry}, 14, 2023.

\bibitem{nemojl}
C.~Fieker, W.~Hart, T.~Hofmann, and F.~Johansson.
\newblock Nemo/hecke: Computer algebra and number theory packages for the julia programming language.
\newblock In {\em Proceedings of the 2017 ACM on International Symposium on Symbolic and Algebraic Computation}, ISSAC '17, pages 157--164, New York, NY, USA, 2017. ACM.

\bibitem{Folland2001}
G.~B. Folland.
\newblock How to integrate a polynomial over a sphere.
\newblock {\em The American Mathematical Monthly}, 108(5):446--448, 2001.

\bibitem{gorlach2019rational}
P.~G{\"o}rlach, E.~Hubert, and T.~Papadopoulo.
\newblock Rational invariants of even ternary forms under the orthogonal group.
\newblock {\em Foundations of Computational Mathematics}, 19(6):1315--1361, 2019.

\bibitem{hong2023owp}
D.~Hong, F.~Yang, J.~A. Fessler, and L.~Balzano.
\newblock Optimally weighted {PCA} for high-dimensional heteroscedastic data.
\newblock {\em SIAM Journal on Mathematics of Data Science}, 5(1):222--250, March 2023.

\bibitem{hotelling1933aoa}
H.~Hotelling.
\newblock Analysis of a complex of statistical variables into principal components.
\newblock {\em Journal of Educational Psychology}, 24(6):417--441, Sept. 1933.

\bibitem{hubert2025algebraically}
E.~Hubert and M.~Jalard.
\newblock Algebraically independent generators for the invariant field of ${S}{O}_3(\mathbb{R})$ and ${O}_3(\mathbb{R})$ representations $\mathbb{R}^3\oplus \mathcal{H}$.
\newblock 2025.

\bibitem{jolliffe2002pca}
I.~T. Jolliffe.
\newblock {\em Principal Component Analysis}.
\newblock Springer-Verlag, New York, NY, 2 edition, 2002.

\bibitem{kogan2023invariants}
I.~A. Kogan.
\newblock Invariants: Computation and applications.
\newblock In {\em Proceedings of the 2023 International Symposium on Symbolic and Algebraic Computation}, pages 31--40, 2023.

\bibitem{neusel2007invariant}
M.~D. Neusel.
\newblock {\em Invariant theory}, volume~36.
\newblock American Mathematical Soc., 2007.

\bibitem{pearson1901lol}
K.~Pearson.
\newblock {LIII}. on lines and planes of closest fit to systems of points in space.
\newblock {\em The London, Edinburgh, and Dublin Philosophical Magazine and Journal of Science}, 2(11):559--572, Nov. 1901.

\bibitem{popov1994invariant}
V.~L. Popov and E.~B. Vinberg.
\newblock Invariant theory.
\newblock In {\em Algebraic Geometry IV: Linear Algebraic Groups Invariant Theory}, pages 123--278. Springer, 1994.

\bibitem{schatz2014est}
M.~D. Schatz, T.~M. Low, R.~A. {van de Geijn}, and T.~G. Kolda.
\newblock Exploiting {{Symmetry}} in {{Tensors}} for {{High Performance}}: {{Multiplication}} with {{Symmetric Tensors}}.
\newblock {\em SIAM Journal on Scientific Computing}, 36(5):C453--C479, Jan. 2014.

\bibitem{schonemann1966gso}
P.~H. Sch{\"o}nemann.
\newblock A generalized solution of the orthogonal procrustes problem.
\newblock {\em Psychometrika}, 31(1):1--10, Mar. 1966.

\bibitem{springer2006invariant}
T.~A. Springer.
\newblock {\em Invariant theory}, volume 585.
\newblock Springer, 2006.

\bibitem{sturmfels2008algorithms}
B.~Sturmfels.
\newblock {\em Algorithms in invariant theory}.
\newblock Springer, 2008.

\end{thebibliography}
